\documentclass[11pt]{amsart}
\usepackage[english]{babel}
\usepackage{amssymb}
\usepackage{hyperref}
\usepackage[initials]{amsrefs}
\usepackage{tikz-cd}
\usepackage{tikz}
\usetikzlibrary{intersections, decorations.markings}
\usepackage{graphicx}
\usepackage{amssymb, mathrsfs, enumerate}
\usepackage{xcolor}

\newcommand{\bbN}{{\mathbb N}}

\newcommand{\bbR}{{\mathbb R}}

\newcommand{\bbZ}{{\mathbb Z}}
\newcommand{\bbC}{{\mathbb C}}
\newcommand{\bbH}{\mathbb{H}}

\newcommand{\calU}{\mathcal{U}}
\newcommand{\calV}{\mathcal{V}}

\newcommand{\bs}{\backslash}
\newcommand{\cyl}{\operatorname{cyl}}

\newcommand{\id}{\operatorname{id}}
\newcommand{\im}{\operatorname{im}}

\newcommand{\isom}{\operatorname{Isom}}

\newcommand{\pr}{\operatorname{pr}}

\newcommand{\vol}{\operatorname{vol}}

\newcommand{\clos}{\operatorname{clos}}

\newcommand{\coker}{\operatorname{coker}}

\newcommand{\SL}{\operatorname{SL}}

\newcommand{\PSL}{\operatorname{PSL}}

\newcommand{\tors}{\operatorname{tors}}

\newcommand{\norm}[1]{{\left\lVert #1\right\rVert}}

\newtheorem{theorem}{Theorem}[section]
\newtheorem{lemma}[theorem]{Lemma}

\newtheorem{corollary}[theorem]{Corollary}
\newtheorem{cor}[theorem]{Corollary}

\newtheorem{prop}[theorem]{Proposition}

\theoremstyle{definition}

\newtheorem{defn}[theorem]{Definition}

\newtheorem{remark}[theorem]{Remark}
\newtheorem{conjecture}[theorem]{Conjecture}

\begin{document}

\title{Homology and homotopy complexity in negative curvature}

\author{Uri Bader}
\address{Weizmann Institute, Israel}
\email{uri.bader@gmail.com}

\author{Tsachik Gelander}
\address{Weizmann Institute, Israel}
\email{tsachik.gelander@gmail.com}

\author{Roman Sauer}
\address{Karlsruhe Institute of Technology, Germany}
\email{roman.sauer@kit.edu}


\begin{abstract}
A classical theorem of Gromov states that the Betti numbers, i.e.~the size of the free part of the homology groups, of negatively curved manifolds are bounded by the volume. We prove an analog of this theorem for the torsion part of the homology in all dimensions $d\ne 3$. Thus the total homology is controlled by the volume. This applies in particular to the classical case of hyperbolic manifolds. 
In dimension $3$ the size of torsion homology cannot be bounded in terms of the volume. 

As a byproduct, in dimension $d\geq 4$ we give a somewhat precise estimate for the number of negatively curved manifolds of finite volume, up to homotopy, and in dimension $d\ge 5$ up to homeomorphism.

These results are based on an effective simplicial thick-thin decomposition which is of independent interest.
\end{abstract}
\maketitle


\section{Introduction}
\subsection{Bounding the homology}
The topological complexity of Hadamard manifolds is controlled, to some extent, by the volume.
This phenomenon is most nicely illustrated in the case of surfaces of constant negative curvature. Indeed, the Gauss--Bonnet theorem implies that the volume coincides (up to a normalization) with the Euler characteristic, which in turn determines the homeomorphism type of the manifold.
In much greater generality, an important theorem of Gromov whose proof was worked out by Ballmann, Gromov and Schroeder in~\cite{BGS} says that the Betti numbers of a negatively\footnote{As noted in \S \ref{sec:nonpositive} below a general version of this theorem holds for analytic non-positively curved manifolds.} curved manifold are bounded by the volume.

\medskip
\noindent
{\bf Notation:} By {\it normalized bounded negative curvature} we mean that the sectional curvature is contained in a closed sub-interval of $[-1,0)$.

\begin{theorem}[\cite{BGS}]\label{thm:BGS}
For every $d\in\mathbb{N}$
there exists $C=C_{d}>0$ such that for every complete $d$-dimensional Riemannian manifold
of normalized bounded negative curvature
and for every degree $k$,
\[ \operatorname{rank} H_k(M;\bbZ)\leq C \vol(M). \]
\end{theorem}

That is, the abelian group $H_k(M;\bbZ)$ is isomorphic to $\bbZ^{b_k}\oplus \tors_k$ where $b_k\le C\vol(M)$ and $\tors_k$ denotes the torsion part. In recent years there has been a growing interest in the size of the torsion part $\tors_k(M)$ motivated by
number theory and topology~\cites{bergeron+venkatesh, lackenby}.
However, $\tors_k$ is much harder to control than $b_k$. 
By a result of Gromov~\cite{Gromov}*{\S1.8} it is known that the group $\tors_k$ is finite. In the same paper Gromov shows that in dimension~$3$ the size of $\tors_1(M)$ cannot be bounded in terms of the volume $M$ (see \S\ref{sec:3d} below for more results in this direction). Our main contribution is to show that in all other dimensions the logarithm of the torsion is linearly bounded by the volume. 

\begin{theorem} \label{thm: main}
For every $d\neq 3$,
there exists $C=C_{d}>0$ such that for every complete $d$-dimensional Riemannian manifold~$M$ of normalized bounded negative curvature
and for every degree $k$,
\[ \log|\tors H_k(M;\bbZ)| \leq C \vol(M). \]
\end{theorem}

\begin{remark}
Note that the theorem is void in the case that $M$ has infinite volume. 
Theorem \ref{thm: main} applies in particular to the classical case of hyperbolic manifolds of dimension $d\ne 3$. Prior to this work no effective bounds were known for the torsion part of the homology of hyperbolic manifolds.
\end{remark}


\subsection{A simplicial thick-thin decomposition}

A crucial ingredient in our proof of Theorem~\ref{thm: main} is an effective simplicial thick-thin decomposition,
which is provided in Theorem~\ref{thm:simplicial-decomposition in intro} below.
In \cite{hv} it is explained how to bound the complexity of the thick part associated with the classical thick-thin decomposition of a non-positively curved manifold.
Theorem~\ref{thm:simplicial-decomposition in intro} below is of the same spirit
and its proof relies on the method and techniques developed in \cite{hv}.
A novel aspect of Theorem~\ref{thm:simplicial-decomposition in intro} is that it provides a control not only of the complexity of the thick part itself, but also on the inclusion of its boundary. This is crucial in the proof of Theorem~\ref{thm: main} in the case of non-compact manifolds. 


A \emph{$(D,V)$-simplicial complex} is a simplicial complex such that the number of its vertices is at most $V$ and the degree of any vertex is at most $D$.
By a \emph{simplicial pair} $(\mathcal{R},\mathcal{R}_0)$ we mean a simplicial complex $\mathcal{R}$ and a subcomplex $\mathcal{R}_0$. That is, every simplex in $\mathcal{R}_0$ also appears in $\mathcal{R}$.
A simplicial pair $(\mathcal{R},\mathcal{R}_0)$ is called a \emph{$(D,V)$-simplicial pair}
if $\mathcal{R}$ is a \emph{$(D,V)$-simplicial complex}.


The following result is a refinement of related results in \cite{hv} (cf. Theorem 7.4 and \S 8 there) in the normalized bounded negative curvature situation. 

\begin{theorem}[see~Theorem~\ref{thm:simplicial-decomposition}]\label{thm:simplicial-decomposition in intro}
Let $d>1$ be an integer. There are constants $D, c>0$ with the following properties. Every complete $d$-dimensional Riemannian manifold~$M$ of normalized bounded negative curvature has a compact $d$-dimensional submanifold $M_+$ with boundary $\partial M_+$ such that the pair $(M_+,\partial M_+)$ is homotopy equivalent to a $(D,c\cdot\vol{M})$-simplicial pair $(\mathcal{R},\mathcal{R}_0)$ and the closure of $M\setminus M_+$ consists of at most $c\cdot\vol(M)$ many connected components, each of which is either homeomorphic to its boundary times $[0,\infty)$ or to a $D^{d-1}$-bundle over $S^1$.
\end{theorem}

\subsection{Counting manifolds}
In \cite{Gromov}, following \cite{Cheeger} and \cite{Margulis}, Gromov shows that for $d\neq 3$ the number of homeomorphism classes of closed $d$-dimensional Riemannian manifolds of normalized bounded negative curvature with volume bounded by some $V>0$ is finite\footnote{This result also applies to the number of diffeomorphism classes of such manifolds. For $d\ne 4$ this is due to the fact that every topological manifold has only finitely many smooth structures, for $d=4$ a separated argument is given in \cite{Cheeger}.}.
In \cite{counting} Burger, Gelander, Mozes and Lubotzky give precise bounds for the homotopy types\footnote{In view of Mostow rigidity theorem, in dimension $>2$, homotopic locally symmetric manifolds are isometric, thus the estimates of \cites{counting,hv} apply also to the number of isometry types.} 
 of hyperbolic manifolds of bounded volume (see also \cite{commensurable} for estimates on the number of manifolds up to commensurability, and fundamental groups up to quasi-isometries).
These were extended to general locally symmetric spaces in~\cite{hv}.
We extend these results to the setting of negatively curved manifolds of variable curvature.
Let $\Gamma_{d}(v)$ denote the number of homotopy classes of complete $d$-dimensional Riemannian manifolds of normalized bounded negative curvature and volume $\le v$. In \S\ref{sub: conclusion of proofs} we prove:
\begin{theorem}\label{thm: counting}
For every $d\geq 4$ there are $\alpha,\beta>0$ such that
\[
 \alpha v\log v\le \log \Gamma_d(v)\le \beta v\log v
\]
for sufficiently large $v>0$.
\end{theorem}

Let $\mathcal{P}_d$ denote the family of complete Riemannian $d$-manifolds of finite volume with normalized bounded negative curvature.
Let $\mathcal{P}_d(v)$ denote the number of homeomorphism classes of elements in $\mathcal{P}_d$ of volume at most~$v$.

Provided $d\ge 5$ two manifolds in $\mathcal{P}_d$ are homotopy equivalent
if and only if they are homeomorphic, by the celebrated work of Farrell--Jones~\cite{borelconj}*{Corollary~7.5} on the Borel conjecture. Hence Theorem~\ref{thm: counting} above implies:

\begin{corollary}\label{cor: homeo counting}
For every $d\geq 5$ there are $\alpha, \beta>0$ such that
\[
 \alpha v\log v\le \log \mathcal{P}_d(v)\le \beta v\log v
\]
for sufficiently large $v>0$.
\end{corollary}
	
\subsection{Some $3$-dimensional examples} \label{sec:3d}
It is well known that Theorem~\ref{thm: main} fails in 
dimension~$3$. 

Already in his 1978 paper~\cite{Gromov}, Gromov obtained a sequence $(M_i)$ of pairwise non-homeomorphic 3-dimensional closed manifolds of bounded negative curvature and bounded volume such that the 
size of $\tors_1(M_i)$ tends to~$\infty$.  
Thurston's theory of $3$-dimensional hyperbolic geometry allows to construct similar examples which are all hyperbolic manifolds.

Theorems~\ref{thm:3D} and~\ref{thm:densetorsion} below nicely illustrate the dramatic failure of Theorem~\ref{thm: main} in dimension~$3$, which is caused by the existence of Dehn fillings. We record these results here because they belong to the complete picture of torsion homology in negative curvature and are not stated explicitly elsewhere. Theorem~\ref{thm:3D} is well known among experts, and we claim no credit. Finally, Theorem~\ref{thm:densetorsion} 
is based on a modification of a construction by Brock and Dunfield~\cite{BrDu}. 

\begin{theorem}\label{thm:3D}
There is a family $(M_{p,q})_{(p,q)\in\bbN^2}$ of pairwise non-homotopy equivalent closed hyperbolic $3$-manifolds $M_{p,q}$ satisfying
\begin{itemize}
\item $\vol(M_{p,q})<2.03$
\item $H_1(M_{p,q})\cong\mathbb{Z}/p\mathbb{Z}$.
\end{itemize}
\end{theorem}

Theorem \ref{thm:3D} reveals two facts. Letting $p\to\infty$, one deduces that the torsion homology cannot be bounded in terms of the volume. Fixing $p$ and letting $q\to\infty$, demonstrates that bounding both the volume and the homology is not enough to impose a finiteness result.

For a fixed symmetric space $S$, it was shown in \cite{7s} that for a sequence of finite volume $S$-manifolds $M_n$ which converges to $S$ in the Benjamini--Schramm topology, the corresponding sequence of normalized Betti numbers $b_k(M_n)/\vol(M_n)$ also converges, and its limit was identified with the $k$-th $L^2$-Betti number of $S$.
In particular, for $S=\mathbb{H}^3$, we get
\[ \lim_{n\to\infty} b_1(M_n)/\vol(M_n)=0. \]
It was speculated that also the normalized torsion of the homology of such a sequence will exhibit a similar phenomenon.

However, the next theorem we discuss shows that the normalized torsion can be unbounded.
In \cite{BrDu}, Brock and Dunfield constructed a sequence of hyperbolic $3$-manifolds which are homology spheres and converges to $\mathbb{H}^3$ in the Benjamini--Schramm topology.
In particular, the torsion vanishes along that sequence.
Building on their example one can construct
a sequence of closed hyperbolic $3$-manifolds which converges to $\bbH^3$ in the Benjamini--Schramm topology and for which the normalized torsion tends to an arbitrary value in $[0,\infty]$.

\begin{theorem} \label{thm:densetorsion}
For every $\alpha\in [0,\infty]$ there exists a sequence of closed hyperbolic 3-manifolds $M^\alpha_n$
which are all rational homology spheres, such that the sequence $M^\alpha_n$
converges in the Benjamini--Schramm topology to $\bbH^3$ and
\[ \lim_{n\to\infty} \frac{\log|\tors~ H_1(M^\alpha_n,\bbZ)|}{\vol(M^\alpha_n)}=\alpha. \]
\end{theorem}

In Theorem~\ref{thm: Benjamini-Schramm and explosive torsion}, which is weaker than Theorem~\ref{thm:densetorsion}, we provide a simpler construction 
of a Benjamini--Schramm convergent sequence of hyperbolic $3$-manifolds with explosive torsion.

The sequences given in the above theorems are non-arithmetic. Moreover, they are not uniformly discrete, i.e. the minimal injectivity radius in $M_n$ tends to $0$. It is conjectured
that the normalized analytic torsion in a residual tower of coverings of a closed hyperbolic $3$-manifold~$M$
converges to $\tau^{(2)}(M)/\vol(M)=\frac{1}{6\pi}$,
where $\tau^{(2)}(M)$ is the $L^2$-torsion of~$M$~\cite{lueck}*{Question~13.73 on p.~487}.
That the normalized size of the torsion in first homology in such a residual tower converges
to the same value is related to the asymptotic vanishing of regulators. Bergeron
and Venkatesh conjecture this to be the case provided $M$ is arithmetic~\cite{bergeron+venkatesh}*{Conjecture~1.3}.
We discuss these issues further in \S\ref{subsection:analytic}.

\subsection{From negative to non-positive curvature}\label{sec:nonpositive}
Recall that the analog of Theorem \ref{thm:BGS} was proved in \cite{BGS} for non-positively curved manifolds without Euclidean factors, under the assumption that the metrics are analytic. 
It is natural to ask whether the results of this paper can be extended as well from the class of negatively curved to the class of non-positively curved manifolds.

By the rank-rigidity theorem~\cites{ballmann, burns+spatzier} an irreducible Hadamard manifold is either of Jacobi rank one or a locally symmetric space of non-compact type of real rank at least two. Jacobi rank one manifolds resemble in many ways negatively curved manifolds and it might be a challenge to generalize our results to that class, and the analysis of \cite{BGS}*{Appendix II} by V.~Schroeder is relevant to that case.
A higher rank complete manifold of finite volume is arithmetic by the Margulis' arithmeticity theorem.
Recall Conjecture 1.3 from~\cite{hv}:

\begin{conjecture}\label{conj:HV}
Let $X$ be a symmetric space of non-compact type. Then there are $\alpha>0$ and $\delta>0$ such that every arithmetic $X$-manifold $M$ is homotopy equivalent to a $(\delta,\alpha\vol(M))$-simplicial complex.
\end{conjecture}

In view of Lemma \ref{lem: torsion bound general from simplicial pair}, Conjecture \ref{conj:HV} implies the conjectured bound on the homology. Recall that Conjecture \ref{conj:HV} was proved in \cite{hv} for non-compact arithmetic manifolds and in particular the homology bounds follow in the non-compact arithmetic case. Moreover, a weak version of Conjecture \ref{conj:HV} was obtained in \cite{hv} for all symmetric spaces with the three exceptions of $X=\mathbb{H}^3,\mathbb{H}^2\times\mathbb{H}^2,\SL_3(\mathbb{R})/\text{SO}(3)$. This weak version is enough to deduce the counting results (see \cite{hv}*{Theorem 1.5, Theorem 1.11} and \cite{vol-vs-rank}*{Corollary 1.4} for the exceptional cases above excluding $\mathbb{H}^3$).

However, to extend Theorem \ref{thm: main} to compact higher rank locally symmetric manifolds by purely geometric and topological means seems a challenging task. In particular, trying to adopt the lines of~\cite{BGS} one encounters many difficulties which do not appear in~\cite{BGS} since they do not have an impact on the rational homology.
More specifically, recall that the basic idea of \cite{BGS} is to define an appropriate function for which one can apply the Morse lemma, and argue by induction on the dimension of the singular set. In our situation, it may very well happen that the singular set has a $3$-dimensional factor in which case we loose control on the torsion completely (cf.~Theorem~\ref{thm:3D}).

\subsection{Lattices in rank 1 simple groups}

Some of the results stated above are novel already in the setting of locally symmetric manifolds.
In particular, the following is an immediate application of Theorem~\ref{thm: main}.

\begin{cor} \label{cor:lattices}
Let $G$ be a connected simple Lie group of real rank 1.
Assume $G$ is not locally isomorphic to $\text{SO}(3,1)$.
Then there exists $C=C_{G}>0$ such that for every torsion free lattice $\Gamma<G$ and for every degree $k$,
\[ \log|\tors H_k(\Gamma;\bbZ)| \leq C \vol(G/\Gamma). \]
\end{cor}

By Theorem~\ref{thm:3D} the restriction that $G$ is not locally isomorphic to $\text{SO}(3,1)$ is indeed necessary.
We note that similar analogues of Theorem~\ref{thm:BGS} and Theorem~\ref{thm: counting} were established for all lattices (i.e, with no torsion-free assumption) in \cite{Samet} and \cite{vol-vs-rank}, respectively.
We do not know if one can omit the assumption that the lattices are torsion-free in Corollary~\ref{cor:lattices}.

\subsection{Structure of the paper}

In the next section we will collect some homotopy theoretical preliminaries.
The technical heart of this paper is \S\ref{sec:hadamard} in which we prove Theorem~\ref{thm:HV}.
Theorem~\ref{thm:simplicial-decomposition in intro} will be proven in \S\ref{sec:mainproof}
and our main results will be proven in \S\ref{sec:hom}. 
Finally, \S\ref{sec:dim3} will be devoted to the study of torsion in dimension~3.

\subsection{Acknowledgment}
We thank Ian Biringer, Yair Minsky and Juan Souto for their advice regarding 3-dimensional geometry. We thank Vikram Aithal and Hartwig Senska for helpful mathematical and stylistic comments. We thank the referee for a detailed and extremely helpful report.  

We are grateful to the organizers of the conference \emph{Manifolds and Groups 2015} in Ventotene where a lot of progress on this project took place. U.B and T.G acknowledge the support of ISF-Moked grant
2095/15. U.B acknowledges the support of ERC grant 306706.

\section{Homotopy-theoretic preliminaries}

\subsection{Some facts about cofibrations}

On a technical level, it will be important for us to extend homotopies from subspaces and to detect homotopy equivalences locally. We collect some well known results about cofibrations that deal with these issues. Recall that a continuous map $i\colon A\to X$ is a \emph{cofibration} if it has the \emph{homotopy extension property} meaning that any continuous map $(X\times \{0\}\cup_{i\times\id} A\times [0,1])\to Y$ can be extended to a continuous map $X\times [0,1]\to Y$. 

\begin{remark}\label{rem: what are cofibrations}
We assume that all spaces are Hausdorff.
Then a cofibration is always an inclusion with closed image. If $A\subset Y$ is
a closed subspace and $(Y,A)$ is an \emph{NDR pair}, i.e.~there
is a continuous function $u\colon Y\to [0,1]$ with $u^{-1}(0)=A$ and a map $H\colon Y\times [0,1]\to Y$ such that
\begin{enumerate}
\item $H(a,t)=a$ for $a\in A$ and $t\in [0,1]$,
\item $H(x,0)=x$ for $x\in Y$,
\item $H(x, 1)\in A$ if $u(x)<1$,
\end{enumerate}
then $A\subset Y$ is a cofibration.
For a reference see~\cite{may}*{Section~6.4}. From that one easily sees that the inclusion of a boundary of a manifold is a cofibration. Another important example of a cofibration is
the inclusion of a simplicial subcomplex~\cite{tomdieck}*{Proposition~8.3.9 on p.~206}.
\end{remark}

The next theorem~\cite{tomdieck}*{Proposition~5.3.4 on p.~112} says that being a homotopy equivalence is a local property in the presence of cofibrations.

\begin{theorem}\label{thm: homotopy equivalence between pushouts}
A commutative diagram
\[
\begin{tikzcd}
          A\arrow{d}{\simeq}\arrow[hookleftarrow]{r} & B\arrow[r]\arrow{d}{\simeq} & C\arrow{d}{\simeq}\\	
          A'\arrow[hookleftarrow]{r} & B'\arrow{r} & C
  \end{tikzcd}
\]
where the hooked arrows are cofibrations and the vertical arrows are homotopy equivalences induces a
homotopy equivalence between the pushouts of the rows.	
\end{theorem}

The next theorem~\cite{tomdieck}*{cf.~Exercise~8 on p.~207} says that the homotopy type of a pushout only depends on the homotopy type
of the attaching map in the presence of a cofibration.

\begin{theorem}\label{thm: homotopy type of pushout}
For $i\in\{1,2\}$ let
\begin{equation*}
		\begin{tikzcd}
				A\arrow[hook]{d}{j}\arrow{r}{g_i} & B\arrow[hook]{d}\\
				X\arrow{r} & Y_i
		\end{tikzcd}
\end{equation*}
be a pushout diagram with a cofibration~$j$. If $g_1$ is homotopic to $g_2$ then $Y_1$ is homotopy equivalent to $Y_2$.
\end{theorem}

The following is a fundamental result in homotopy theory which is also used in the proof of the previous two theorems.

\begin{theorem}[\cite{may}*{Proposition on p.~47}]\label{thm: cofibrations and homotopy equivalences}
If in a commutative square
\begin{equation*}
		\begin{tikzcd}
				A\arrow[hook]{d}\arrow{r}{\simeq} & B\arrow[hook]{d}\\
				X\arrow{r}{\simeq} & Y
		\end{tikzcd}
	\end{equation*}
the horizontal maps are homotopy equivalences and the vertical maps are cofibrations then $(X,A)$ and $(Y,B)$ are homotopy equivalent as pairs of spaces.
\end{theorem}

\begin{lemma}\label{lem: making diagram commutative}
	Assume $j\colon A\to X$ is a cofibration, and that the left square below
		commutes up to homotopy. Then there is a map $F'\colon X\to Y$ homotopic to $F$ such that the right square below
		(strictly) commutes.
		\begin{equation*}
		\begin{tikzcd}
				A\arrow[hook]{d}{j}\arrow{r}{f} & B\arrow{d}{g}\\
				X\arrow{r}{F} & Y
		\end{tikzcd}\quad\quad
		\begin{tikzcd}
				A\arrow[hook]{d}{j}\arrow{r}{f} & B\arrow{d}{g}\\
				X\arrow{r}{F'} & Y
		\end{tikzcd}
	\end{equation*}
	In particular, if $F$ is a homotopy equivalence then $F'$ is one as well.
\end{lemma}

\begin{proof}
Let $h\colon A\times [0,1]\cup X\times\{0\}\to Y$ be a homotopy from
$F\circ j$ to $g\circ f$ on $A\times [0,1]$ and $F$ on $X\times\{0\}$. By the cofibration property $h$ extends to a homotopy $H\colon X\times [0,1]\to Y$. Set $F':=H_1$.
\end{proof}

\subsection{Revisiting the nerve construction}\label{sub: revisiting nerve}
A family of subsets of a space $Y$ is called an \emph{open cover} of $Y$ if
the union of their interiors cover $Y$.
We recall two constructions of spaces associated to an open cover $\calU=\{U_i\}_{i\in I}$ of a space~$Y$.
This material is needed in the final step of the
proof of Theorem~\ref{thm:HV} in \S\ref{sub: proof of thm:HV}.

The \emph{nerve} of $\calU$ is the simplicial complex $N(\calU)$ whose simplices correspond to tuples in $\calU$ with nonempty intersection, i.e. every set of $\calU$ corresponds to a vertex and two sets form an edge if they intersect, etc. More precisely, $N(\calU)$ is the quotient
space of the disjoint union of copies of $\Delta^n$ ranging over $n\ge 0$ and the
subsets $\{i_0,\ldots, i_n\}\subset I$ of cardinality $n+1$ such that $U_{i_0}\cap\ldots \cap U_{i_n}\ne\emptyset$ with identifications
over the faces of $\Delta^n$ corresponding to dropping an element
in $\{i_0,\ldots, i_n\}$.

In a similar spirit, one defines the space
$Y_\calU$ as in Hatcher's book~\cite{hatcher}*{Section~4G} as the quotient space of the disjoint union of products
\[\bigl(U_{i_0}\cap\ldots\cap U_{i_n}\bigr)\times \Delta^n\]
ranging over $n\ge 0$ and the
subsets $\{i_0,\ldots, i_n\}\subset I$ of cardinality $n+1$ such that   $U_{i_0}\cap\ldots \cap U_{i_n}\ne\emptyset$. The identifications in the quotient are
over the faces of $\Delta^n$ corresponding to dropping an element
in $\{i_0,\ldots, i_n\}$.

Next we discuss some easy functorial properties. Let $A\subset Y$ be a subspace. Let $J\subset I$ be a subset, and
let $\calV=\{V_j\}_{j\in J}$ be an open cover of $A$ such that $V_j\subset U_j$ for every $j\in J$.
The inclusion maps
\[ V_{j_0}\cap\ldots V_{j_n}\times \Delta^n\to U_{j_0}\cap\ldots U_{j_n}\times \Delta^n \]
are compatible with the identification and induce a map
\begin{equation}\label{eq: from homotopy colimit to homotopy colimit}
A_\calV\to Y_\calU.
\end{equation}
Similarly we get an embedding of simplicial complexes
\begin{equation}\label{eq: from nerve to nerve}
	 N(\calV)\to N(\calU).
\end{equation}
Further the left factor projection
\[\bigl(U_{i_0}\cap\ldots\cap U_{i_n}\bigr)\times \Delta^n\to U_{i_0}\cap\ldots U_{i_n}\subset Y\]
and the right factor projection
\[\bigl(U_{i_0}\cap\ldots\cap U_{i_n}\bigr)\times \Delta^n\to \Delta^n\]
induce maps
\begin{equation}\label{eq: comparision maps}
N(\calU)\leftarrow Y_\calU\rightarrow Y.
\end{equation}

Recall that an open cover $\calU$ of a topological space $Y$ is called \emph{good} if every nonempty intersection of sets of the cover is contractible.
The following result is taken from Hatcher's book~\cite{hatcher}*{Section~4.G}
in the case that the elements of $\calU$ are open. Hatcher does not assume that $\calU$ is locally finite.
An open cover in our sense (i.e.~the interiors of the elements of $\calU$ cover the space but the elements are not required to be open) is numerable
on paracompact spaces provided it is locally finite, and one can appeal to~\cite{tomdieck}*{Section~13.3} for the general case.

\begin{remark}
If the cover in question consists of open sets rather than sets whose interiors cover the space the assumption of local finiteness can be dropped in the following results.
\end{remark}

\begin{theorem}\label{thm: htp equivalence for nerve of good cover}
Let $\calU$ be a locally finite open cover of a paracompact space~$Y$.
Then $Y_\calU\to Y$ is a homotopy equivalence. If, in addition, $\calU$ is good then
also $Y_\calU\to N(\calU)$ is a homotopy equivalence.
\end{theorem}

Clearly, we obtain the following well known result as a consequence~\citelist{\cite{hatcher}*{Corollary 4G.3.}\cite{tomdieck}*{Theorem 13.3.1 on p.~324}}.

\begin{corollary}[Nerve lemma]
The nerve of a locally finite good open cover of a paracompact space $Y$ is homotopy equivalent to $Y$.
\end{corollary}

Next we record a relative version of the previous result. The method of its proof will be used in a slightly more complicated context in Step 4 of the proof of Theorem~\ref{thm:HV}.

\begin{lemma}[Relative nerve lemma]\label{lem: relative nerve lemma}
Let $A\subset X$ be a subspace of a paracompact Hausdorff space~$X$ such that $A\hookrightarrow X$ is a cofibration. Let $\calU=\{U_i\}_{i\in I}$ be a locally finite open cover of $X$. Let $J\subset I$ be a subset, and
let $\calV=\{V_j\}_{j\in J}$ be a locally finite open cover of $A$ such that $V_j\subset U_j$ for each  $j\in J$.
    If $\calU$ and $\calV$ are good covers of $X$ and $A$, respectively, then $(X,A)$ and $(N(\calU), N(\calV))$ are homotopy equivalent as pairs.
\end{lemma}

\begin{proof}
	Consider the following commutative diagram involving the maps~\eqref{eq: from homotopy colimit to homotopy colimit}, \eqref{eq: from nerve to nerve} and~\eqref{eq: comparision maps}.
	\begin{equation*}
		\begin{tikzcd}
		A\arrow[d, hook] & A_\calV\arrow[d]\arrow[l, "\simeq"'] \ar[r, "\simeq"] & N(\calV)\arrow[d, hook]\\
		X& X_\calU\arrow[l, "\simeq"]\arrow[r,"\simeq"'] & N(\calU)
		\end{tikzcd}
	\end{equation*}
	By Remark~\ref{rem: what are cofibrations} and Theorem~\ref{thm: htp equivalence for nerve of good cover} the hooked arrows are cofibrations and horizontal arrows are homotopy equivalences. 	
	Denote the middle vertical map by~ $j$. We do not know whether $j$
	is a cofibration in general. But we can replace $j$ by a cofibration via the mapping cylinder $\cyl(j)$, i.e.~we obtain a factorization of $j$
	\[ A_\calV\hookrightarrow \cyl(j)\xrightarrow{\simeq} X_\calU\]
	into a cofibration and a homotopy equivalence. So we obtain a commutative diagram with horizontal homotopy equivalences and vertical cofibrations: 	\begin{equation*}
		\begin{tikzcd}
	     A\arrow[d,hook] & A_\calV\arrow[d,hook]\arrow[l,"\simeq"'] \arrow[r,"\simeq"] & N(\calV)\arrow[d,hook]\\
		X& \cyl(j)\arrow[l,"\simeq"']\arrow[r,"\simeq"] & N(\calU)
		\end{tikzcd}
	\end{equation*}
	By Theorem~\ref{thm: cofibrations and homotopy equivalences} the left and the middle vertical inclusions are homotopy equivalent as pairs of spaces. Similarly, the left and the right vertical inclusions are homotopy equivalent as pairs of spaces. So the same holds for the left and right vertical inclusion which finishes the proof.
\end{proof}

\section{A result for Hadamard manifolds} \label{sec:hadamard}


The main result of this section, Theorem~\ref{thm:HV}, is a refinement of a variant of \cite[Theorem 7.4]{hv}.
First we will introduce some notation and conventions that will be used throughout the paper.

For a Riemannian manifold $M$ and a point $x\in M$ we denote the injectivity radius at $x$ of $M$ by $\text{InjRad}_M(x)$.
For $\epsilon>0$ we let
\[
 M_{\ge\epsilon}:=\{x\in M:\text{InjRad}_M(x)\ge\epsilon/2\}~\text{and}~M_{<\epsilon}:=M\setminus M_{\ge\epsilon}.
\]
These are called the {\em $\epsilon$-thick part} and the {\em $\epsilon$-thin part} of $M$, respectively.
For a subset $A\subset M$ and $\tau>0$ we let 
\[(A)_\tau = \{ x\in M : d(x,A)\le \tau\}\]
denote the $\tau$-neighborhood of $A$. 
Fixing a universal cover $X$ of $M$, for a subset $A\subset M$ we denote by $\tilde A$ the pre-image in $X$ under the covering map\footnote{Caution: $\tilde A$ does not denote a universal cover of $A$.}.
With regard to the deck transformation action of $\Gamma=\pi_1(M)$ on $X$ the set $\tilde A$ is $\Gamma$-invariant.
For $\gamma\in\Gamma$ we denote by $d_\gamma$ the displacement function $x\mapsto d(x,\gamma\cdot x)$
and let $\{d_\gamma<\epsilon\}=\{x:d_\gamma(x)<\epsilon\}$ denote its $\epsilon$-sub-level set. 
Recall that a {\em Hadamard space} is a simply connected complete Riemannian manifold of non-positive sectional curvature.
If $X$ is a Hadamard space then the metric $d:X\times X\to [0,\infty)$ is convex and smooth outside the diagonal. Therefore
the $\epsilon$-sub-level sets are convex and, by the implicit function theorem, have smooth boundary.
Finally, observe the equation
$$\tilde M_{<\epsilon}=\cup_{\gamma\in\Gamma\setminus\{1\}}\{d_\gamma<\epsilon\}.$$

\begin{theorem}\label{thm:HV}
Let $M$ be a $d$-dimensional complete Riemannian manifold of finite volume with sectional curvature varying in $[-1,0]$. Let $X$ be its universal covering
which is a Hadamard manifold. Let $\Gamma=\pi_1(M)$, so $M=\Gamma\backslash X$.
Fix constants $0<\epsilon<\epsilon_0$.
Let $M_+\subset  M_{\ge\epsilon}$ be a compact $d$-dimensional submanifold with boundary.
Let $M_-=M\setminus M_+$.
Suppose that there is a conjugation invariant subset $\Gamma'\subset \Gamma$ and a conjugation invariant  assignment of numbers $\{\epsilon_\gamma\}_{\gamma\in\Gamma'}\subset [\epsilon,\epsilon_0]$ such that
the pre-image $\tilde M_-\subset X$ of $M_-$ in $X$ is given by
$\tilde M_-=\cup_{\gamma\in\Gamma'}\{ d_\gamma<\epsilon_\gamma\}$.
Furthermore, suppose that for every point $x\notin \tilde M_-$ the group
$$
 \langle \gamma\in \Gamma':d_\gamma(x)\le 4\epsilon_\gamma\rangle
$$
is commutative.
Then $(M_+,\partial M_+)$ is homotopy equivalent to a $(D,c\cdot\vol{M_+})$ simplicial pair, where the constants $D,c>0$ depend only on $d$, $\epsilon_0$ and $\epsilon$.
\end{theorem}

The goal of this section is a proof of Theorem~\ref{thm:HV}.
We start with a preliminary subsection and then conclude the proof of Theorem~\ref{thm:HV} in four steps.
The proof relies heavily on results and constructions given in \S3, \S4 and \S7 of~\cite{hv}.
We recall that in \cite{hv} it is globally assumed that $X$ is a symmetric space of non-compact type.
However, \S3, \S4, \S7  only rely on the assumption that $X$ is a Hadamard manifold. Accordingly, we are free to use the results of these sections in~\cite{hv}.

\subsection{An estimate for essential volume}\label{sub: essential volume}

We define
\begin{equation}\label{eq: multiplicity function}
N(d,r,R):=\vol_{\mathbb{H}^d}B(o,R+r/2)/\vol_{\mathbb{E}^d}(B(0,r/2))
\end{equation}
as the ratio of volumes of the corresponding balls in the hyperbolic $d$-space $\mathbb{H}^d$ and the Euclidean $d$-space $\mathbb{E}^d$,
respectively.

\begin{lemma}\label{lem:ess-vol}
Let $X$ be a $d$-dimensional Hadamard manifold with sectional curvature varying in~$[-1,0]$.
For $0<r<R$ and any $x\in X$, the cardinality of an $r$-discrete subset of $B_X(x,R)$ is bounded above by $N(d,r,R)$.
\end{lemma}

Indeed, the $r/2$ balls centered at the points of such an $r$-discrete set are disjoint and, on the other hand, contained in the $(R+r/2)$ ball around $x$. Therefore, the lemma is a straightforward consequence of the following volume estimate which follows by combining the Bishop--Gunther (lower bound) and the Bishop--Gromov (upper bound) comparison theorems:

\begin{theorem}[\cite{bishop+gromov}*{Theorem~3.101 on p.~169 and Theorem~4.19 on p.~214}]\label{thm:comparison}
Let $X$ be a Hadamard manifold of dimension $d$ and sectional curvature varying in the segment $[-1,0]$. For every  $R>0$ and every $x\in X$ we have
\[ \vol_{\mathbb{E}^d}(B(0,R)) \leq \vol_X(B(x,R)) \leq \vol_{\mathbb{H}^d}B(o,R). \]
\end{theorem}

\begin{corollary}\label{cor:2.5}
Let $M=\Gamma\backslash X$ be a $d$-dimensional complete Riemannian manifold of sectional curvature varying in the segment $[-1,0]$. Let $\epsilon>0$ and $x\in \widetilde M_{\ge\epsilon}$. Then for every $R>0$ we have:
$$
 \#\{\gamma\in\Gamma: d_\gamma(x)\le  R\}\le N(d,\epsilon,R).
$$
\end{corollary}

\begin{proof}
Indeed, the set $\{\gamma\cdot x: \gamma\in \Gamma,~d_\gamma(x)\le R\}$ is $\epsilon$-discrete and contained in the ball ${B}(x,R)$.
\end{proof}

\subsection{Proof of Theorem~\ref{thm:HV}}\label{sub: proof of thm:HV}

\subsection*{Step 1: Obtuse angles between sub-level sets.}

The goal of the first step is to show the following variant of \cite{hv}*{Lemma 7.1}.
 
\begin{lemma}\label{lem: angle of sublevel sets}
Let $x\in (\tilde M_-)_\epsilon\setminus\tilde M_-$. Consider the commutative subgroup
\[A:=\langle \gamma\in \Gamma':d_\gamma(x)\le 4\epsilon_\gamma\rangle<\Gamma.\]
Then we have
\[
 \nabla d(\cdot,\{d_\alpha<\epsilon_\alpha\})|_x \cdot \nabla d(\cdot,\{d_{\beta}<\epsilon_{\beta}\})|_x  \geq 0
\]
for every $\alpha,\beta\in A\cap\Gamma'$.
\end{lemma}

The special case where $x\in \partial \{d_\alpha<\epsilon_\alpha\}\cap\partial \{d_\beta<\epsilon_\beta\}$ follows directly from \cite[Lemma 7.1]{hv}. However, we will need the result for points which are not necessarily on the boundary. The proof relies on tools from \cite[\S 7]{hv}.

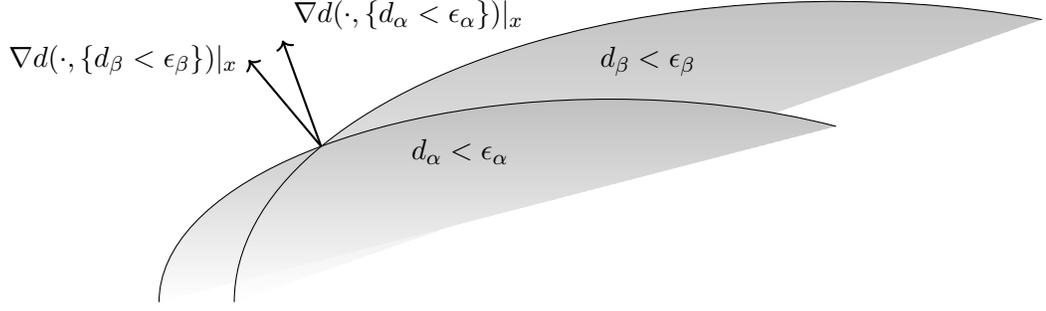
\begin{figure}[h]\label{fig:commuting}
    \centering
         \begin{tikzpicture}
     \shade [top color=gray!50] (0,0) arc (-180:-290:8 and 4);
     \shade [top color=gray!50] (-1,0) arc (-180:-300:6 and 2.7);

    \draw [name path=first] (0,0) arc (-180:-290:8 and 4) ;
    \draw [name path=second] (-1,0) arc (-180:-300:6 and 2.7);
    \path [name intersections={of=first and second, by=x}];
    \draw[->, thick] (x) -- +(130:1.5) node [left] {$\nabla d(\cdot,\{d_\beta<\epsilon_\beta\})|_x$};
    \draw[->, thick] (x) -- +(110:1.5) node [above right] {$\nabla d(\cdot,\{d_{\alpha}<\epsilon_{\alpha}\})|_x$};

    \node at (3,2) {$d_\alpha<\epsilon_\alpha$};
    \node at (5.5,3.2) {$d_\beta<\epsilon_\beta$};

    \end{tikzpicture}
\caption{If $\alpha$ and $\beta$ commute, the angle formed by the sub-level sets is obtuse.}

    \label{angles}
\end{figure}

\begin{remark}
Note that $A$ is not empty.
Indeed, because of $x\in (\tilde M_-)_\epsilon\setminus\tilde M_-$ we have $d(x,\tilde M_-)\leq\epsilon$,
so there exists at least one $\gamma\in\Gamma'$ such that
\begin{equation}\label{eq: distant from sublevel set}
x\in (\{d_\gamma<\epsilon_\gamma\})_\epsilon \subset \{d_\gamma<\epsilon_\gamma+2\epsilon\}
\subset \{d_\gamma<4\epsilon_\gamma\}.
\end{equation}
\end{remark}

The basic fact that stands behind 
Lemma \ref{lem: angle of sublevel sets}, is given in the following:

\begin{lemma}\label{lem:big-angle}
Let $X$ be a Hadamard manifold, $\Omega\subset X$ a convex subset with smooth boundary, and $\alpha$ an isometry of $X$ which preserves $\Omega$. Let $x\in \partial \Omega$ be a point on the boundary of $\Omega$, let $\hat f\in T_x(X)$ be the external normal to $\Omega$ at $x$. Then $\nabla d_\alpha |_x\cdot \hat f\ge 0$.
\end{lemma}

Note that $\hat f=\nabla d(\cdot,\Omega)|_x$ and that the vectors $\nabla d_\alpha |_x$ and $\nabla d(\cdot,\{d_\alpha\le d_\alpha(x)\})|_x$ only differ by a positive scalar.

\begin{proof}[Proof of Lemma \ref{lem:big-angle}]
Let $\pi_\Omega$ be the nearest point projection from $X$ to $\Omega$. Let $c(t)$ be the geodesic ray starting at $x$ with $\dot{c}(0)=\hat f$, and let $c'(t)=\alpha\cdot c(t)$
(see Figure \ref{fig:distance}).
Note that $\pi_\Omega(c(t))=x$ and $\pi_\Omega(c'(t))=\alpha\cdot x$. Since the sectional curvature is non-positive $\phi(t):=d(c(t),c'(t))$ is a convex function. Since $\pi_\Omega$ is also a contraction $\phi(t)$ is monotonically non-decreasing. Thus,
\begin{multline*}
 0\le \frac{d}{dt}\phi(t)|_{t=0}=\frac{d}{dt}d(c(t),\alpha\cdot c(t))|_{t=0}\\=\frac{d}{dt}d_\alpha(c(t))|_{t=0}=\nabla d_\alpha|_x\cdot \dot{c}(0)=
 \nabla d_\alpha|_x\cdot\hat{f}.
\end{multline*}

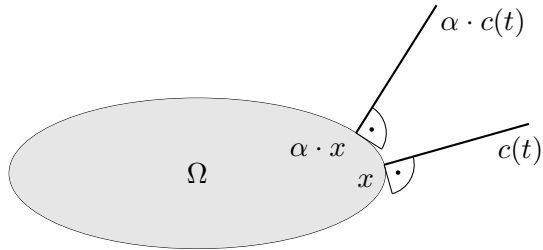
\begin{figure}[h]\label{fig:distance}
    \centering
\begin{tikzpicture}[
    tangent/.style={
        decoration={
            markings,
            mark=
                at position #1
                with
                {
                    \coordinate (tangent point-\pgfkeysvalueof{/pgf/decoration/mark info/sequence number}) at (0pt,0pt);
                    \coordinate (tangent unit vector-\pgfkeysvalueof{/pgf/decoration/mark info/sequence number}) at (1,0pt);
                    \coordinate (tangent orthogonal unit vector-\pgfkeysvalueof{/pgf/decoration/mark info/sequence number}) at (0pt,1);
                }
        },
        postaction=decorate
    },
    use tangent/.style={
        shift=(tangent point-#1),
        x=(tangent unit vector-#1),
        y=(tangent orthogonal unit vector-#1)
    },
    use tangent/.default=1
]

       \draw[tangent=0.01, tangent=0.06] (0,0) ellipse [x radius=2.5,  y radius=1];
       \fill [ color=gray!20] (0,0) ellipse [x radius=2.5,  y radius=1];
       \draw[use tangent] (0,0)--(-0.4,0);
       \draw[thick, use tangent] (0,0)--(0,-2);
       \draw[use tangent] (-0.4,0) arc [start angle=180, end angle=270, radius=0.4];
       \fill[use tangent] (-0.15, -0.15) circle [radius=0.03];
       \node[use tangent, below left] (0,0) {$x$};
       \draw[use tangent=2] (0,0)--(-0.4,0);
       \draw[thick, use tangent=2] (0, 0)-- (0,-2);
     \draw[use tangent=2] (-0.4,0) arc [start angle=180, end angle=270, radius=0.4];  
     \fill[use tangent=2] (-0.15, -0.15) circle [radius=0.03];
            \node[use tangent=2, below left] (0,0) {$\alpha\cdot x$};

    \node at (0,0) {$\Omega$};
    \node at (4.3,0.3) {$c(t)$};
    \node at (3.8,2) {$\alpha\cdot c(t)$};
\end{tikzpicture}
\end{figure}
\end{proof}

%
%
%

\begin{lemma}\label{clm3.8}
Let $\Theta\subset X$ be a closed subset. Let $t\ge 0$. 
For a point
$z$ outside the interior of $(\Theta)_t$ we have
\begin{align}\label{eq: neighborhood of convex set}
d(z, \Theta)=d(z, (\Theta)_t)+t,	
\end{align}
and
\begin{align}\label{eq: gradient of neigborhoods of convex sets}
\nabla d(\cdot, (\Theta)_t)|_z=\nabla d(\cdot, \Theta)|_z.
\end{align}
\end{lemma}

\begin{proof}
The second assertion~\eqref{eq: gradient of neigborhoods of convex sets} follows from the first~\eqref{eq: neighborhood of convex set}, which we show next. 
Let $z_\Theta$ be a point in $\Theta$ of minimal distance to $z$. Let $c$ be a minimizing geodesic from $z$ 
to $z_\Theta$. Let $z_t$ be the first point of intersection of $c$ with $(\Theta)_t$. Then $d(z_t, z_\Theta)=t$ and $d(z,z_\Theta)=d(z, z_t)+d(z_t, z_\Theta)$. If $d(z, (\Theta)_t)<d(z, z_t)$ then there 
would be a point $z_t'\in (\Theta)_t$ with $d(z,z_t')<d(z, z_t)$ and thus $d(z, \Theta)\le d(z,z_t')+t<d(z,z_t)+t=d(z,\Theta)$, which is absurd. Hence $d(z, (\Theta)_t)=d(z, z_t)$ which implies~\eqref{eq: neighborhood of convex set}. 
\end{proof}

\begin{proof}[Proof of Lemma~\ref{lem: angle of sublevel sets}]
Denote $t_\alpha=d(x,\{d_\alpha\le \epsilon_\alpha\})$ and $t_\beta=d(x,\{d_\beta\le \epsilon_\beta\})$. Without loss of generality suppose $t_\alpha\le t_\beta$. Set
$$
 \Omega:=\overline{(\{d_\beta\le\epsilon_\beta\})_{(t_\beta - t_\alpha)}}.
$$
unless $t_\alpha=t_\beta$ in which case we set $\Omega=\{d_\beta\le\epsilon_\beta\}$. Since $\alpha$ and $\beta$ commute, $\Omega$ is $\alpha$-invariant.  Lemma~\ref{lem:big-angle} implies that 
\[\nabla d_\alpha |_y\cdot \nabla d(\cdot,\Omega)|_y\ge 0\]
for every $y\in\partial\{d_\alpha\le\epsilon_\alpha\}\cap\partial\Omega$. 
Using the remark following Theorem~\ref{lem:big-angle} and 
Lemma~\ref{clm3.8} we obtain that 
\[
 \nabla d(\cdot,\{d_\alpha\le\epsilon_\alpha\})|_y\cdot \nabla d(\cdot,\{ d_\beta\le\epsilon_\beta\})|_y\ge 0
\]
for every $y\in\partial\{d_\alpha\le\epsilon_\alpha\}\cap\partial\Omega$. 
Denoting the maximal angle by
\[
 \varphi_0=\sup\{\angle \big( \nabla d(\cdot,\{d_\alpha\le\epsilon_\alpha\})|_y,\nabla d(\cdot,\Omega)|_y\big) : y\in\partial\{d_\alpha\le\epsilon_\alpha\} \cap\partial\Omega\}
\]
we deduce that $\varphi_0\le\frac{\pi}{2}$. Now for $t>0$ we let
\[
 \varphi_t=\sup\{\angle \big( \nabla d(\cdot,\{d_\alpha\le\epsilon_\alpha\})|_y,\nabla d(\cdot,\Omega)|_y\big) : y\in\partial(\{d_\alpha\le\epsilon_\alpha\})_t \cap\partial(\Omega)_t\}.
\]
Observe that since the curvature is non-positive both $\{d_\alpha\le\epsilon_\alpha\}$ and $\Omega$ are convex, so we may apply
\cite[Lemma 7.2]{hv} and deduce that $\varphi$ is non-increasing. By assumption,
\[
 d(x,\{d_\alpha\le\epsilon_\alpha\})=d(x,\Omega)=t_\alpha.
\]
It follows that
\[
  \angle \big( \nabla d(\cdot,\{d_\alpha\le\epsilon_\alpha\})|_x,\nabla d(\cdot,\Omega)|_x\big)\le\varphi_{t_\alpha}\le\frac{\pi}{2}.
\]
Therefore
\[
\nabla d(\cdot,\{d_\alpha\le\epsilon_\alpha\})\cdot\nabla d(\cdot,\{d_\beta\le\epsilon_\beta\})=
 \nabla d(\cdot,\{d_\alpha\le\epsilon_\alpha\})|_x\cdot\nabla d(\cdot,\Omega)|_x\ge 0.\qedhere
\]
\end{proof}

\subsection*{Step 2: Setting the constants $b$ and $\delta$.}

In this step we define constants $b=b(d)$ and $\delta=\delta(d,b,\epsilon)$
which will be used throughout the proof and discuss their properties.

\begin{defn}
	For every dimension $d$ let $b=b(d)$ be the maximal cardinality of a $1$-discrete subset of unit vectors in $\bbR^d$. For every dimension $d$ and every $\epsilon>0$ choose $\delta_0(\epsilon/4,b+1,d)>0$ that satisfies the 	statement of Lemma~\ref{lem: positiveD} below. Then we define $\delta=\min\{\delta_0(\epsilon/4,b+1,d),\epsilon/(2(b+1))\}$.
\end{defn}

\begin{lemma}[cf.~\cite{hv}*{Proposition~4.7}]\label{lem: positiveD}
For all $\alpha>0$, $\beta>0$ and for every dimension $d>0$
there exists $\delta_0=\delta_0(\alpha,\beta,d)>0$
 satisfying the following property:
For $\delta\leq \delta_0$ and every $d$-dimensional Hadamard manifold $X$ with sectional curvature $\ge -1$,
any point $x\in X$, any ball $C\subset X$ of radius~$\alpha$ with $x\in\partial C$, and
any point $y\in B(x, \delta\beta)\setminus
(C)_{\delta/2}$, the inner product of the external normal vector of $C$ at $x$
with the tangent at $x$ to the geodesic segment $[x,y]$
is positive.
\end{lemma}

\begin{proof}
The statement is a variant of~\cite{hv}*{Proposition~4.7}. For convenience we sketch the proof.
Assume the contrary. Then there is a sequence $\delta_n \to 0$,
$d$-dimensional Hadamard manifolds $X_n$,
points $x_n\in X_n$ and $y_n\in B(x_n, \delta_n\beta)\setminus ( C_n)_{\delta_n/2}$,
and $\alpha$-balls $C_n$ violating the above statement.

Let $d_n$ be the metric of $X_n$. Let $B_n:=B(x_n, \delta_n\beta)$. The sequence of pointed metric spaces $(B_n, x_n)$ with scaled metric
$\frac{1}{\delta_n\beta}d_n$ converges in the Gromov--Hausdorff topology to the (pointed) standard Euclidean unit ball $(B_0,0)$.
The sequence $(C_n\cap B_n, x_n)$ converges to a half ball $C_0$ in $B_0$ because the scaled radius of $C_n$ is at least $\alpha/\delta_n\beta$, thus tends to~$\infty$. Upon passing to a subsequence, we may also assume that $(y_n)$
converges to a point $y_0\in B_0\setminus (C_0)_{\frac{1}{2\beta}}$.
The inner product of $y_0$ and the external normal to $C_0$ is $\ge \frac{1}{2\beta}>0$ which contradicts the assumption.
\end{proof}

Our choice of $b$ was made so that the following statement holds true.

\begin{lemma}[{\cite[Proposition~4.2]{hv}}]\label{closeN}
For every $s>0$ and $t>0$ with $s+t< \epsilon$, we have
\[
M\setminus (M_-)_s \subset ( M\setminus (M_-)_{s+t})_{bt}.\]
\end{lemma}

\begin{proof}
Up to different notation the lemma is identical with \cite[Proposition~4.2]{hv}.
Note that \cite[Proposition~4.2]{hv} is a step in the proof of  \cite[Lemma~4.1]{hv} and some assumptions of the former are implicit in its formulation --- they are stated explicitly in the formulation of the latter.
However, not all the assumptions of the latter are actually used in the proof of the former. Hence our task is to compare the setting of~\cite{hv}*{Lemma~4.1 and~4.2} with
the present one. The correspondence\footnote{Recall that the results of \S4 of \cite{hv} apply to a general Hadamard manifold.} is
$S=X$, $M'=M_+$, $I=\Gamma'$, $X_\gamma=\{d_\gamma<\epsilon_\gamma\}$ for $\gamma\in \Gamma'$, $\tau=s$, and $\delta=t$.
Then we have $M'=M_+\subset M_{\geq\epsilon}$, and the cover of $\tilde M_+$ by the $X_i$ is indeed a locally finite cover by convex open sets with a smooth boundary.
Finally, Lemma~\ref{lem: direction} below provides the (not necessarily continuous) vector field $\hat n$ with the exact property assumed in~\cite[Lemma~4.1]{hv}. The properties above are the only ones used in the proof of~\cite[Proposition~4.2]{hv}. Thus we get the result.
\end{proof}

\begin{lemma}\label{lem: direction}
For every $x\in (\tilde M_-)_\epsilon\setminus\tilde M_-=\tilde M_+\cap (\tilde M_-)_\epsilon$ there is a unit vector $\hat n(x)\in T_xX$ such that
\[
 d_\gamma(x)\le 4\epsilon_\gamma
 ~ \Rightarrow ~
 \hat n(x) \cdot \nabla d(\cdot,\{d_{\gamma}<\epsilon_{\gamma}\})|_x > 1/b.
\]
for every $\gamma\in \Gamma'$.
\end{lemma}

\begin{proof}
Let $I:=\{\gamma\in \Gamma':d_\gamma(x)\le 4\epsilon_\gamma\}$. Let $\Delta(x)$
be a maximal $1$-discrete subset of the vectors $\nabla d(\cdot,\{d_\gamma<\epsilon_\gamma\})|_x$, $\gamma\in I$. The set $\Delta(x)$ has at most
$b=b(d)$ elements. Let $\hat n(x)$ be the sum of the vectors of $\Delta(x)$ normalized to unit length. The statement about $\hat n(x)$ is then easily verified
(see~\cite{hv}*{Lemma~7.3}).
\end{proof}

\subsection*{Step 3: Constructing a deformation retract}

We define
\[
 \tilde M_+'=X\setminus (\tilde M_-)_{\frac{\epsilon}{2}}~\text{ and }M_+'=\Gamma\backslash\tilde M_+'.
\]
The main task of this step is to show that the shrinking $M_+'$ of $M_+$ has
the same homotopy type as $M_+$. Indeed there is a strong
deformation retract from $M_+$ to $M_+'$ with nice properties.
Recall that a \emph{strong deformation retract} from a space $W$ to a subset $A\subset W$ is a homotopy $h\colon W\times [0,1]\to W$ from $\id_W$ to a map into $A$ such that $h_t\vert_A=\id_A$ for every $t\in [0,1]$.

\begin{lemma}\label{lem: deformation retract via vector field}
There exists a
strong deformation retract $f\colon M_+\times [0,1]\to M_+$ from $M_+$ to $M_+'$ with the following
properties:
\begin{enumerate}
\item[(i)] The time 1 map, $f_1\colon M_+\to M_+'$, restricts to a homeomorphism $f_1\colon \partial M_+\to \partial M_+'$.
\item[(ii)] For every $y\in M_+'$ with $d(y,(M_-)_{\frac{\epsilon}{2}})>\delta$, the ball $B(y,(b+1)\delta)$ is stable under $f$ in the sense that if $(x,t)\in  M_+\times [0,1]$ and
$f(x,t)\in B(y,(b+1)\delta)$ then $f(x,s)\in B(y,(b+1)\delta)$ for all $s\in [t,1]$.
\end{enumerate}
\end{lemma}

Lemma \ref{lem: deformation retract via vector field} could be deduced from the analysis in \cite[\S 3]{hv}. However, we cannot refer directly to statements from \cite[\S 3]{hv} but require adjustments in the arguments. Therefore we shall give a self contained proof which is in a sense more straightforward.

\begin{center}
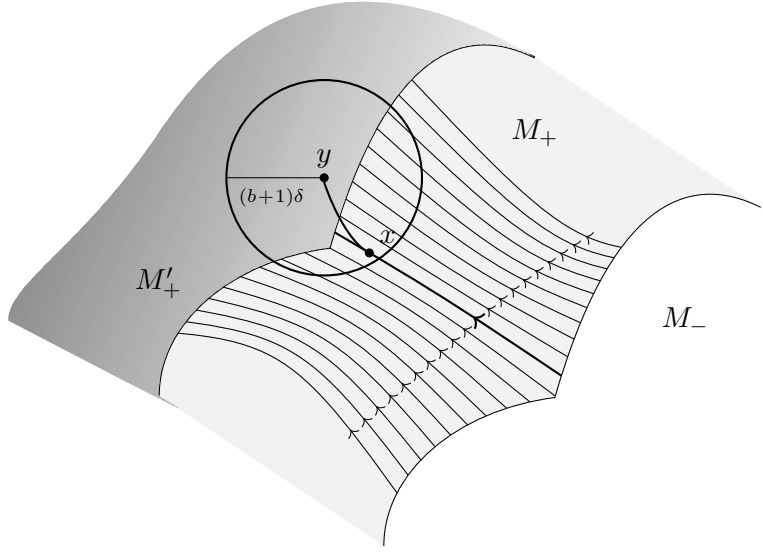
\begin{figure}[h]\label{fig: deformation retract}
\begin{tikzpicture}
\fill[color=gray!10, xshift=-3cm, yshift=+2cm] (0,0) arc (180:100:2.75cm and 2cm)  .. controls (3,4.5) and (4,5) .. (5,4.5) -- (8, 2.5) -- (3, -2) -- cycle;


\foreach \s in {0.6, 0.8, 1.0,1.2,1.4,1.6, 1.8, 2.0, 2.2,2.4, 2.8, 3.0, 3.2,3.4,3.6,3.8,4.0,4.2,4.4}
      \draw[thin, xshift=\s cm, yshift=\s cm,decoration={markings, mark=at position 0.7 with {\arrow{<}}},
        postaction={decorate}] (-4.5,2.3) .. controls (-1.9, 2.5-0.7*\s+0.1) .. (0,-0.5);
 \draw[thick, xshift=2.6 cm, yshift=2.6 cm,decoration={markings, mark=at position 0.7 with {\arrow{<}}},
        postaction={decorate}] (-4.5,2.3) .. controls (-1.9, 2.5-0.7*2.6+0.1) .. (0,-0.5);


\draw[thick, xshift=-3cm, yshift=+2cm] (0,0) arc (180:100:2.75cm and 2cm)  .. controls (3,4.5) and (4,5) .. (5,4.5);

\draw[thick] (0,0) arc (180:100:2.75cm and 2cm) coordinate (midpoint) .. controls (3,4.5) and (4,5) .. (5,4.5);

\fill[color=white]  coordinate (startpoint) ++(4,0) -- (0,0) arc (180:100:2.75cm and 2cm) .. controls (3,4.5) and (4,5) .. (5,4.5) .. controls (4,0) .. (startpoint);

\coordinate (endpoint) at (-0.8,4.9);

\shade[shading angle = 70, shading = axis, color=gray, xshift=-3cm, yshift=+2cm] coordinate (startpoint) (0,0) arc (180:100:2.75cm and 2cm)  .. controls (3,4.5) and (4,5) .. (5,4.5) .. controls (2.5, 6) and (1,5).. (0,3.5) .. controls (-1,2) and (-2, 1.5) .. (-2,1) ..controls (0.5,-0.3 ) .. cycle;


\draw[thick] (endpoint) circle (1.3cm);


%

\draw[thick]  (endpoint) ++(+0.6,-1) coordinate (start_geo) .. controls +(-0.4,0.2) and +(-0.1, 0.2) .. (endpoint);

\filldraw (start_geo) circle (1.5pt) node[above right] {$x$};

\filldraw (endpoint) circle (1.5pt) node[above] {$y$};


\draw[thin] (endpoint) -- +(-1.3,0);

\draw (endpoint) ++(-0.7,0) node[below] {\tiny $(b\negthinspace +\negthinspace 1)\delta$};

\node at (4,3) {$M_-$};

\node at (2,5.5) {$M_+$};

\node at (-3,3.5) {$M_+'$};
	
\end{tikzpicture}
\caption{Deformation retract by (truncated) flow lines. The angle between the flow line through $x$ and the geodesic $[x,y]$ is acute. The distance of $y$ to the light gray part is $>\delta$.}
\end{figure}
\end{center}

\begin{proof}
We will construct a $\Gamma$-equivariant strong deformation
retract from $\tilde M_+$ to $\tilde M_+'$; by passing to $\Gamma$-quotients we obtain the desired map~$f$. To this end,
we will construct a
smooth $\Gamma$-equivariant vector field $\Vec{V}$ on
$X$ with the following properties.

\begin{enumerate}[a)]
\item The image of the support of $\Vec{V}$ in $M=\Gamma\bs X$ is compact. Hence $\Vec{V}$ possesses a global flow $c\colon X\times\bbR\to X$.
\item On $\{(x,t)\in X\times [0,\infty)\mid d(c(x,t), \tilde M_-)\in (0,\epsilon/2)\}$
the function
$t\mapsto d(c(x,t), \tilde M_-)$ has lower Lipschitz constant~$1$.
\item Let $y\not\in (\tilde M_-)_{\epsilon/2+\delta}$. For $x\in B(y,(b+1)\delta)\cap  (\tilde M_-)_{\epsilon/2}$ the
angle between $\Vec{V}(x)$ and the geodesic from $x$ to $y$ is acute.
\end{enumerate}

Before constructing $\Vec{V}$ let us explain how, given $\Vec{V}$, we obtain the desired deformation
retract.

For $x\in X$ with $d(x, \tilde M_-)\le\epsilon/2$ define
$T(x)$ to be the minimal time $t\in [0,\infty)$ such that $d(c(x,T(x)), \tilde M_-)=\epsilon/2$, that is, $T(x)$ is the entrance time of the flow
line through $x$ into $\tilde M_+'$.
We claim that the entrance time is continuous.
Because of (b) we have $T(x)\le \frac{\epsilon}{2}$. Let $\epsilon_1>0$. The function
\[g\colon\bigl\{x\in X: d(x, \tilde M_-)\le\epsilon/2\bigr\}\times [0, \frac{\epsilon}{2}]\to\bbR,~(x,t)\mapsto d(c(x,t), \tilde M_-)\]
is continuous. Thus every point whose distance to $\tilde M_-$ is at most $\epsilon/2$ has a neighborhood $U$
so that there is $\delta_1>0$ such that $d(x,x')<\delta_1$ with $x,x'\in U$ implies
$|g(x,t)-g(x',t)|<\epsilon_1$ for every $t\in [0,1]$. Consider $x,x'\in U$ with
$d(x,x')<\delta_1$. Assume that $T(x)\le T(x')$.
Then $d(c(x', T(x)), \tilde M_-)\ge \epsilon/2-\epsilon_1 $ and thus
$T(x)-T(x')\le \epsilon_1$ according to (b).
Thus indeed, $T$ is continuous.

The map $f\colon \tilde M_+\times [0,1]\to \tilde M_+$ defined by
\[ f(x,t)=\begin{cases}
       	      c(x, tT(x)) & \text{ if $d(x, \tilde M_-)\le \epsilon/2$, }\\
       	      x & \text{ otherwise,}
\end{cases}
 \]
is the desired $\Gamma$-equivariant strong deformation retract from $\tilde M_+$ to $\tilde M_+'$. It is clear that $f_1$ restricts
to a map $\partial \tilde M_+\to \partial \tilde M_+'$. The inverse
of $f_1$ is obtained by flowing along $-\Vec{V}$ to $\partial \tilde M_+$ in a similar way as above.

The stability of balls, i.e.~the fact that flow lines do not leave a ball $B(y, (b+1)\delta))$ once they enter it, follows directly from property (c).

Next we turn to the construction of~$\Vec{V}$.
For $\gamma\in \Gamma'$ we define the function
\[\phi_\gamma(x)=d(x, \{d_\gamma<\epsilon_\gamma\})\]
which is convex and smooth.
It satisfies the
equivariance property
\[\phi_\gamma(\lambda x)=\phi_{\lambda \gamma\lambda^{-1}}(x).\]
We choose smooth
functions $\tau_\gamma\colon X\to [0,1]$ such that $\tau_\gamma \equiv 1$ on
$(\{d_\gamma<\epsilon_\gamma\})_{\epsilon/2}\cap \tilde M_+$ and $\tau_\gamma$ is supported
in $(\{d_\gamma<\epsilon_\gamma\})_{\epsilon/2+\delta/2}\cap (\tilde M_+)_{\delta/2}$,
and the collection $\{\tau_\gamma\}$ is $\Gamma$-equivariant, i.e. 
$\tau_\gamma(\lambda x)=\tau_{\lambda \gamma\lambda^{-1}}(x),~\forall \gamma,\lambda\in \Gamma$.
Now we set
\[ \Vec{V}(x)=\sum_{\gamma\in \Gamma'}\tau_\gamma(x)\cdot \nabla \phi_\gamma\vert_x.\]
Since for any $r>0$ and $x\in X$ the subset $\{g\in\isom(X): d_g(x)<r\}$ is compact and $\Gamma<\isom(X)$ is discrete, the above sum only involves finitely many summands.
Since $M$ has finite volume, Theorem~\ref{thm:comparison} easily implies that $M_+$ is compact. 

Property (b) follows since for every $x$ with $d(x, \tilde M_-)\in (0,\epsilon/2]$ and for every $\gamma\in\Gamma'$ with
\begin{equation}\label{eq: realizing dist}
	 d(x, \tilde M_-)=d\bigl(x, \{d_\gamma<\epsilon_\gamma\}\bigr),
\end{equation}
we have
\[ \langle \Vec{V}(x), \nabla\phi_\gamma\vert_x\rangle \ge \langle \nabla\phi_\gamma\vert_x, \nabla\phi_\gamma\vert_x\rangle =1\]
by Lemma~\ref{lem: angle of sublevel sets} and~\eqref{eq: distant from sublevel set}.
Let $y\not\in (\tilde M_-)_{\epsilon/2+\delta}$. Let $\Vec{U}(x)\in T_xX$ be the tangent to the geodesic from $x$ to $y$ in $X$. The stability
of $B(y, (b+1)\delta)$ will follow once we show:
\begin{equation}\label{eq: implication on inner product}
x\in  B(y,(b+1)\delta)\cap  (\tilde M_-)_{\epsilon/2}
~\Rightarrow~ \Vec{V}(x)\cdot \Vec{U}(x) >0.
\end{equation}
Let $x\in B(y,(b+1)\delta)\cap  (\tilde M_-)_{\frac{\epsilon}{2}}$. The  implication~\eqref{eq: implication on inner product} will follow from
\begin{equation}\label{eq: acute angle}
 \nabla \phi_\gamma|_x\cdot \Vec{U}(x)>0~~\text{ for every $x,\gamma$ with $\phi_\gamma(x)\le \epsilon/2+\delta/2$}.
\end{equation}
Let $z$ be a point on the shortest geodesic from $x$ to $\{d_\gamma\leq\epsilon_\gamma\}$ satisfying $d(z,x)=\epsilon/4$. Let $\pi_\gamma$ be the nearest point projection to $\{d_\gamma\leq\epsilon_\gamma\}$.
We let $C=B(z,\epsilon/4)$. For $w\in C$ note that
\[d(w, \pi_\gamma(x))\le d(w, z)+d(z, \pi_\gamma(x))\le \epsilon/4+d(x,\pi_\gamma(x))-\epsilon/4=\phi_\gamma(x).\]
Hence $C\subset\{\phi_\gamma\le \phi_\gamma(x)\}$. Hence
$\nabla \phi_\gamma|_x$ is the external normal vector of $C$ at $x$.
Now~\eqref{eq: acute angle} follows by applying Lemma~\ref{lem: positiveD} for
$\alpha=\epsilon/4$ and $\beta=b+1$ provided
$d(y,C)>\delta/2$. Since $\phi_\gamma(x)\le \epsilon/2+\delta/2$ we have
\[d(z, \tilde M_-)\le \phi_\gamma(z)= \phi_\gamma(x)-\epsilon/4<\epsilon/4+\delta/2.\]
So
\[ d(y, C)>d(y, \tilde M_-)-(\epsilon/4+d(z, \tilde M_-))>\delta/2
\]
since $y\not\in (\tilde M_-)_{\epsilon/2+\delta}$.
\end{proof}

\subsection*{Step~4: Completion of the proof of Theorem~\ref{thm:HV}} In the final step of the proof of
Theorem~\ref{thm:HV} we construct suitable open covers of $M_+$ and $\partial M_+$ that yield
the desired simplicial pair via the nerve construction. Let $\pr\colon X\to M$
be the universal cover. We define
\begin{itemize}
\item $Z$ to be a maximal $\delta$-discrete subset in $M\setminus (M_-)_{\frac{\epsilon}{2}+\delta}$ and for every $z\in Z$ the point $\bar z\in X$ to be a choice of lift of $z$, 
\item $N_+$ to be the union of the balls $B(z,(b+1)\delta)$ with $z\in Z$,
\item $N_0:=N_+\cap \overline{(M_-)_{\epsilon/2}}$,
\item $W\subset Z\times\Gamma'$ to be the subset of pairs 
 $(z,\gamma)$ for which
\[V_{z,\gamma}:=\pr\bigl(B(\bar z,(b+1)\delta)\cap \overline{\{d_\gamma<\epsilon_\gamma\})_{\epsilon/2}}\bigr)\ne \emptyset,\]
\item $\calV$ to be the family of sets $V_{z,\gamma}$ indexed over $W$,
\item $\calU$ to be the family of sets
\[\bigl(B(z, (b+1)\delta)\bigr)_{z\in Z}\quad\mbox{and}\quad \bigl(V_{z,\gamma}\bigr)_{(z,\gamma)\in W}\]
indexed over $Z\sqcup W$.
\end{itemize}

The family $\calV$ is an open cover of $N_0$ in the sense of
\S\ref{sub: revisiting nerve}, that is, the sets in $\calV$ are not necessarily open in $N_0$ but the union
of their $N_0$-interiors cover $N_0$. The family $\calU$ is an open cover of $N_+$.
Both $\calV$ and $\calU$ are good since 
their elements are convex.

In the sequel we refer to the deformation retract $f$ constructed in the previous step. The deformation retract at time $t$ is denoted by $f_t$.

Let $z\in Z$. Since
$(b+1)\delta<\epsilon/2$ and $d(z, M_-)\ge \epsilon/2+\delta$
we get $B(z,(b+1)\delta)\cap M_-=\emptyset$, thus $B(z,(b+1)\delta)\subset M_+$ for every $z\in Z$. Hence
\[ N_+\subset M_+.\]
By Lemma~\ref{closeN} we have
\[ M_+'\subset N_+\]
and hence also $\partial M_+'\subset N_0$.


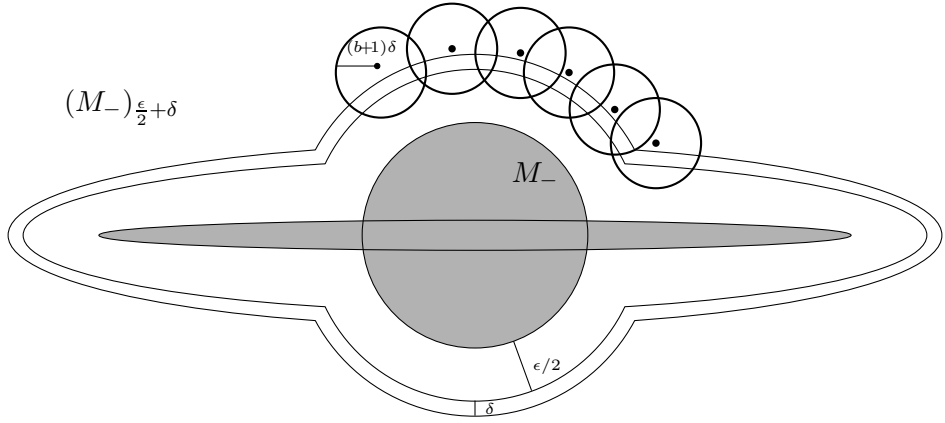
\begin{figure}[h]
\centering
\begin{tikzpicture}


    \draw[thick] (0,0) circle (2.4cm);
	\draw[thick] (0,0) ellipse (6.2cm and 1.2cm);
    \fill[white] (0,0) circle (2.4cm);
    \fill[white] (0,0) ellipse (6.2cm and 1.2cm);
    \draw[thick] (0,0) circle (2.2cm);
	\draw[thick] (0,0) ellipse (6cm and 1cm);
    \fill[white] (0,0) circle (2.2cm);
    \fill[white] (0,0) ellipse (6cm and 1cm);

	\fill[color=gray!60] (0,0) circle (1.5cm);
	\fill[color=gray!60] (0,0) ellipse (5cm and 0.2cm);
	\draw (0,0) circle (1.5cm);
	\draw (0,0) ellipse (5cm and 0.2cm);

    \draw[thick] (27: 2.7cm) circle (0.6cm);
    \filldraw[thick] (27: 2.7cm) circle (1pt);

	\draw[thick] (42: 2.5cm) circle (0.6cm);
	\filldraw[thick] (42: 2.5cm) circle (1pt);

    \draw[thick] (60:2.5cm) circle (0.6cm);
    \filldraw[thick] (60:2.5cm) circle (1pt);

    \draw[thick] (76:2.5cm) circle (0.6cm);
    \filldraw[thick] (76:2.5cm) circle (1pt);

    \draw[thick] (97:2.5cm) circle (0.6cm);
    \filldraw[thick] (97:2.5cm) circle (1pt);

    \draw[thick] (120:2.5cm) circle (0.6cm);
    \filldraw (120:2.6cm) circle (1pt);


    \draw (120:2.6cm) -- +(-0.55cm, 0) node[above right] { $\scriptscriptstyle (b\!+\!1)\delta$};

    \draw (-70:1.5cm) -- (-70:2.2cm) node [midway, right] { $\scriptscriptstyle \epsilon/2$};

    \draw (-90:2.2cm) -- (-90:2.4cm) node [midway, right] { $\scriptscriptstyle \delta$};


    \node at (0.8,0.8) {$M_-$};

    \node at (160:5cm) {$(M_-)_{\frac{\epsilon}{2}+\delta}$};
\end{tikzpicture}
    \caption{A schematic picture of the situation in the proof of Theorem~\ref{thm:HV}.}
\end{figure}

\begin{lemma}
The inclusions $M_+'\hookrightarrow N_+$ and $\partial M_+'\hookrightarrow N_0$
are homotopy equivalences.
\end{lemma}

\begin{proof}
We claim that
\[ r\colon N_+\hookrightarrow M_+\xrightarrow{f_1} M_+'\]
is a homotopy inverse of the inclusion $j\colon M_+'\hookrightarrow N_+$. It is clear that $r\circ j=\id$. By the stability of balls under $f$ (see Lemma~\ref{lem: deformation retract via vector field}) we
have \[f_t(B((b+1)\delta, z))\subset B((b+1)\delta, z)~\text{ for every $t\in [0,1]$}\] since this is obviously true  for $t=0$. In particular, $f_t$ preserves $N_+$ for every $t\in [0,1]$. Thus
$f$ restricts to a homotopy between $\id$ and $j\circ r\colon N_+\to N_+$. Since $f_t$ also restricts to $\overline{(M_-)_{\epsilon/2}}$ we may argue similarly for $\partial M_+'\hookrightarrow N_0$.
\end{proof}

Consider the following commutative diagram
\begin{equation}\label{eq: main diagram for nerve}
		\begin{tikzcd}
		\partial M_+\arrow[hook]{d}\arrow{r}{\cong}[swap]{f_1\vert_{\partial M_+}} & \partial M_+'\arrow[hook]{d}\ar[hook]{r}{\simeq} & N_0\arrow[hook]{d}{j} & (N_0)_\calV\arrow{l}[swap]{\simeq}\arrow[hook]{d}{k}\arrow{r}{\simeq} & N(\calV)\arrow[hook]{d}\\
		M_+\arrow{r}{\simeq}[swap]{f_1}& M_+'\arrow[hook]{r}{\simeq} & N_+ & (N_+)_{\calU}\arrow{l}[swap]{\simeq}\arrow{r}{\simeq} & N(\calU)
		\end{tikzcd}
	\end{equation}
where $\cong$ indicates a homeomorphism, $\simeq$ indicates a homotopy equivalence and the hooked arrows are inclusions. The horizontal maps on the right hand side are homotopy equivalences since the open covers $\calV$ and $\calU$ are good (see Theorem~\ref{thm: htp equivalence for nerve of good cover}).
It was shown in Step~$3$ that its restriction $f_1\colon \partial M_+\to \partial M_+'$ is a homeomorphism, in particular, a homotopy equivalence. The outer vertical maps
in~\eqref{eq: main diagram for nerve} are cofibrations, being the inclusion
of the boundary of a manifold and the inclusion of a subcomplex, respectively. We avoid proving that the inner vertical maps are cofibrations, in which case we could appeal directly to the relative nerve lemma as a blackbox (see Lemma~\ref{lem: relative nerve lemma}). Instead we will use below the method of the proof of the relative nerve lemma.

We factorize $j$ into a cofibration $j'$ and a homotopy equivalence $r$ via the
mapping cylinder $\cyl(j)$ of $j$:
\[j\colon N_0\xrightarrow{j'}\cyl(j)\xrightarrow{r} N_+\]
We choose a homotopy inverse $r'\colon N_+\to \cyl(j)$ of~$r$.
Next consider the following three squares:
\begin{equation*}
\begin{tikzcd}
     \partial M_+\arrow[r, "\simeq"]\arrow[d,hook] &N_0\arrow[d, "j",hook]\\	
     M_+\arrow[r, "F"',"\simeq"] & N_+
\end{tikzcd}\quad
\begin{tikzcd}
     \partial M_+\arrow{r}{\simeq}\arrow[hook]{d} &N_0\arrow[hook]{d}{j'}\\	
     M_+\arrow{r}[swap]{r'\circ F}{\simeq} & \cyl(j)\quad
\end{tikzcd}\quad
\begin{tikzcd}
     \partial M_+\arrow{r}{\simeq}\arrow[hook]{d} &N_0\arrow[hook]{d}{j'}\\	
     M_+\arrow{r}[swap]{F'}{\simeq} & \cyl(j)\quad
\end{tikzcd}
\end{equation*}
The left square arises from composing the left two squares in~\eqref{eq: main diagram for nerve} and thus commutes. The middle
square, and in particular its lower horizontal arrow, arises from composition with $r'$. The middle square only commutes up to homotopy. From an application of Lemma~\ref{lem: making diagram commutative} we obtain the commutative right square. By Theorem~\ref{thm: cofibrations and homotopy equivalences} applied to the right square we obtain a homotopy equivalence of pairs
\begin{equation}\label{eq: first pair homotopy equivalence}
	 (M_+, \partial M_+)\simeq (\cyl(j), N_0).
\end{equation}
Next we apply the mapping cylinder construction to the map $k$ in~\eqref{eq: main diagram for nerve} and get a factorization of $k$
as
\[ (N_0)_\calV\xrightarrow{k'}\cyl(k)\xrightarrow{s} (N_+)_{\calU}.\]
By pre-composing the lower horizontal map in the third square of~\eqref{eq: main diagram for nerve} with the homotopy equivalence $s$ and post-composing it with $r'$ we obtain a homotopy-commutative
diagram with vertical cofibrations:
\[\begin{tikzcd}
 N_0\arrow[hook]{d}{j'} & (N_0)_\calV\arrow{l}[swap]{\simeq}\arrow[hook]{d}{k'}\\
		 \cyl(j) & \cyl(k)\arrow{l}[swap]{\simeq}
\end{tikzcd}
\]
As above, we can use Lemma~\ref{lem: making diagram commutative} to replace its lower horizontal arrow by another homotopy equivalence making the square strictly commute and then apply
Theorem~\ref{thm: cofibrations and homotopy equivalences} to get a homotopy equivalence of pairs
\begin{equation}\label{eq: second pair homotopy equivalence}
 (\cyl(j), N_0)\simeq (\cyl(k), (N_0)_\calV). 	
\end{equation}
By composing the lower horizontal arrow in the right hand square of~\eqref{eq: main diagram for nerve} with $s$ we obtain the commutative diagram
\[\begin{tikzcd}
     (N_0)_\calV\arrow[hook]{d}{k'}\arrow{r}{\simeq} & N(\calV)\arrow[hook]{d}\\
     \cyl(k)\arrow{r}{\simeq} & N(\calU)
\end{tikzcd}\]
to which we apply Theorem~\ref{thm: cofibrations and homotopy equivalences}. We thus obtain a homotopy equivalence of
pairs
\begin{equation}\label{eq: third pair homotopy equivalence}
	(\cyl(k), (N_0)_\calV)\simeq (N(\calU), N(\calV)).
\end{equation}
The desired homotopy equivalence of pairs
\begin{equation}\label{eq: final pair homotopy equivalence}
	(M_+,\partial M_+)\simeq (N(\calU), N(\calV))
\end{equation}
is obtained by combining~\eqref{eq: first pair homotopy equivalence}
and~\eqref{eq: second pair homotopy equivalence} and~\eqref{eq: third pair homotopy equivalence}.

Referring to~\eqref{eq: multiplicity function} we define two constants:
\begin{align*}
c&:=\bigl(N(d,\epsilon, 4\epsilon_0)+1\bigr)/\vol_{\mathbb{E}^d}\bigl(B(0,\delta/2)\bigr)\\
D&:=\bigl(N(d,\epsilon, 4\epsilon_0)+1\bigr)N(d,\delta,2(b+1)\delta)	
\end{align*}

The proof of Theorem~\ref{thm:HV} will be completed by the following lemma.

\begin{lemma}
The nerve $N(\calU)$ is a $(D,c\vol(M_+))$-simplicial complex. 	
\end{lemma}

\begin{proof}
First we show that the number of vertices in $N(\calU)$, that is $\# Z+\#W$, is at most $c\vol(M_+)$.
Since $Z$ is a $\delta$-discrete subset of $M_+$,
Theorem~\ref{thm:comparison} implies that
 \[\# Z\leq \vol(M_+)/\vol_{\mathbb{E}^d}(B(0,\delta/2)).\]
Let $z\in Z$. If $\gamma\in\Gamma'$ is such that $B(\bar z, (b+1)\delta)$ intersects $\overline{(\{d_\gamma<\epsilon_\gamma\})_{\epsilon/2}}$ then
\[
 d_\gamma(\bar z)\le\epsilon_\gamma+2(\frac{\epsilon}{2}+(b+1)\delta)<4\epsilon_0.
\]
Hence Corollary~\ref{cor:2.5} implies that
\begin{equation}\label{eq: multiplicity with fixed ball}
 \#\bigl\{\gamma\in\Gamma'~:~B(\bar z, (b+1)\delta)\cap \overline{(\{d_\gamma<\epsilon_\gamma\})_{\frac{\epsilon}{2}}}\ne\emptyset\bigr\}\le
N(d,\epsilon, 4\epsilon_0).
\end{equation}
Thus $\#W \le \#Z\cdot N(d,\epsilon, 4\epsilon_0)$
from which the desired bound $\#Z+\#W\le c\vol(M_+)$ follows. \medskip\\
Next we show that the degree of a vertex in $N(\calU)$ is at most~$D$. Consider an element $U_0\in\calU$
which represents a vertex in $N(\calU)$. Let $z_0\in Z$ or $(z_0,\gamma_0)\in W$, respectively, be the index of the element $U_0\in\calU$.
The degree of $U_0$ is bounded
using~\eqref{eq: multiplicity with fixed ball} by
\begin{align*}
\deg U_0 \le & \#\bigl\{z\in Z : B(z_0, (b+1)\delta)\cap B(z, (b+1)\delta)\ne\emptyset\bigr\}\\&
	+\#\bigl\{ (z,\gamma)\in Z\times \Gamma' : B(\bar z_0, (b+1)\delta)\cap \overline{\{d_\gamma<\epsilon_\gamma\})_{\frac{\epsilon}{2}}}\ne\emptyset~\text{ and }\\
 &\omit\hfill $B(z_0, (b+1)\delta)\cap B(z, (b+1)\delta)\ne\emptyset	\bigr\}$\\
\le & \bigl(N(d,\epsilon, 4\epsilon_0)+1\bigr)\cdot \#\bigl\{z\in Z : B( z_0, (b+1)\delta)\cap B(z, (b+1)\delta)\ne\emptyset\bigr\}.
\end{align*}
The latter factor is bounded the cardinality of $Z\cap B(z_0, 2(b+1)\delta)$. With Lemma~\ref{lem:ess-vol}
the desired bound $\deg U_0\le D$ follows.
\end{proof}

\section{A simplicial decomposition of negatively curved manifolds} \label{sec:mainproof}

The main goal of this section is to prove the following theorem. 

\begin{theorem}[see Theorem~\ref{thm:simplicial-decomposition in intro}]\label{thm:simplicial-decomposition}
Let $d>1$ be an integer. There are constants $D, c>0$ with the following properties. Every complete $d$-dimensional Riemannian manifold~$M$ of normalized bounded negative curvature has a compact $d$-dimensional submanifold $M_+$ with boundary $\partial M_+$ such that the pair $(M_+,\partial M_+)$ is homotopy equivalent to a $(D,c\cdot\vol{M})$-simplicial pair $(\mathcal{R},\mathcal{R}_0)$ and the closure of $M\setminus M_+$ consists of at most $c\cdot\vol(M)$ many connected components, each of which is either homeomorphic to its boundary times $[0,\infty)$ or to a $D^{d-1}$-bundle over $S^1$. \end{theorem}

A similar result, without the information on $\partial M^+$ is given in \cite[\S 4,7,8]{hv}. The preparation made in the previous section allows us to control the inclusion of the boundary $\partial M^+$ in $M^+$. An additional complication comes from the fact that we deal with general negatively curved manifolds whose universal cover is not necessarily homogeneous. This is the reason we will not work with the usual thick-thin decomposition but with a variant which gives `more weight' to loops corresponding to elements with larger centralisers.  

\begin{remark}\label{rem: boundary restriction}
Note that the connected components of $\partial M_+$ correspond to the connected components of $\mathcal{R}_0$.
 In particular, if $N\subset\partial{M}_+$ is a union of several connected components then $(M_+,N)$ is homotopy equivalent to a $(D,c\cdot\vol{M})$ simplicial pair $(\mathcal{R},\mathcal{R}_0')$. 
\end{remark}

\subsection{The Margulis lemma}\label{sub: margulis}
Fix $d>1$. Let $X$ be a $d$-dimensional Hadamard space and $\Gamma$ a discrete group of isometries of $X$. For $\gamma\in\Gamma$ we denote by $d_\gamma$  the  displacement function $x\mapsto d(x,\gamma\cdot x)$. For $\epsilon>0$ and $x\in X$ we let
$$
 \Sigma_{x,\epsilon}=\{\gamma\in \Gamma\setminus\{1\}:d_\gamma(x)<\epsilon\},~\text{and}~\Gamma_{x,\epsilon}=\langle\Sigma_x\rangle.
$$

In negative curvature, the classical Margulis lemma has the following form:

\begin{theorem}[Margulis lemma {\cite[10.3]{BGS}}]\label{thm: margulis lemma}
Given $d>1$ there are $\epsilon(d)>0$ and $m=m(d)\in \mathbb{N}$ such that if $X$ is a $d$-dimensional Hadamard manifold with normalized bounded negative curvature then for every torsion-free discrete group $\Gamma$ of isometries of $X$, any $\epsilon\leq \epsilon(d)$ and every $x\in X$, the group $\Gamma_{x,\epsilon}$ is either
\begin{itemize}
\item trivial,
\item a cyclic group generated by a loxodromic element, or
\item consisting of parabolic elements sharing a common fixed point at $X(\infty)$ and admitting a normal nilpotent subgroup of index $\le m$.
\end{itemize}
\end{theorem}

We shall refer to the second and third types as `loxodromic' and `parabolic' types respectively.
We shall refer to the constant $\epsilon(d)$ as the \emph{Margulis  constant}.


\subsection{Nilpotent subgroups of small elements}\label{sec:normal-nilpotent}
Let $X$ be a $d$-dimensional Hadamard space of normalized bounded negative curvature and
let $\epsilon\le\epsilon(d)$.

Let $\Gamma$ be a torsion free discrete group of isometries of $X$.
For every $x\in X$ we denote by $N_{x}=N_{x,\epsilon}$ the (unique) maximal normal nilpotent subgroup of $\Gamma_{x}=\Gamma_{x,\epsilon}$.
It satisfies 
\[
 N_x=\prod \{ H:  H\lhd\Gamma_{x}~\text{is normal nilpotent of index}~\le m\}.
\]
The group~$N_x$ is a nilpotent characteristic subgroup of $\Gamma_{x}$ of index $\le m$ satisfying $N_{\gamma\cdot x}=\gamma N_x\gamma^{-1}$ for every $x\in X$
and $\gamma\in\Gamma$.

When $\Gamma_x$ is of  loxodromic type we have that $N_x=\Gamma_x$. When $\Gamma_x$ is of parabolic type, it fixes a unique point $\zeta\in X(\infty)$ and preserves the horospheres around it.
In that case, we set
$$
 \Gamma_\zeta:=\langle \gamma\in\Gamma:\gamma\cdot\zeta=\zeta\rangle,
$$
We shall call such $\zeta$ a $\Gamma$-{\it parabolic} point.

Suppose now that $M=\Gamma\backslash X$ has finite volume.
Given a $\Gamma$-parabolic point $\zeta\in X(\infty)$ and fixing a horosphere $H$ centered at $\zeta$ we get that $\Gamma_\zeta$ acts cocompactly on $H$. It follows that $\Gamma_\zeta$ is finitely generated.
In particular, we see that if $c(t)$ is a geodesic in $X$ with $c(\infty)=\zeta$ then for all $t$ sufficiently large we have $\Gamma_{c(t)}=\Gamma_\zeta$ and hence we may set $N_{c(t)}=N_\zeta$. Thus $N_\zeta$ is a characteristic subgroup of $\Gamma_\zeta$ of index at most $m$. It follows in particular that $N_{\gamma\cdot\zeta}=\gamma N_\zeta\gamma^{-1}$ for every $\gamma\in \Gamma$ and $\zeta\in X(\infty)$ $\Gamma$-parabolic.
Moreover, we deduce that $N_\zeta$ acts cocompactly on $H$ and hence its cohomological dimension is $d-1$. Since the cohomological dimension of a torsion free nilpotent group coincides with the Hirsch rank ~\cite{bieri}*{Theorem~ 7.10}, and since the nilpotency degree is obviously bounded by the Hirsch rank, we obtain:

\begin{lemma}
Given a $\Gamma$-parabolic point $\zeta\in X(\infty)$,
the nilpotency degree of $N_\zeta$ is at most $d-1$.
\end{lemma}



\subsection{The thick-thin decomposition}\label{sec:thick-thin}

We will require the following variant of the classical thick-thin decomposition due to Thurston.

\begin{theorem} \label{thm:think-thin}
Fix $0<\epsilon \leq \epsilon(d)$.
Assume $M=\Gamma\backslash X$ is a complete $d$-manifold of normalized bounded negative curvature. Suppose further that $M$ has finite volume.
Suppose that we are given an assignment of numbers $\gamma\mapsto\epsilon_\gamma,~\gamma\in\Gamma\setminus\{1\}$, which is conjugation invariant and valued in the interval $[\epsilon,\epsilon(d)]$. Let
$$
 \tilde M_-:=\cup_{\gamma\in\Gamma\setminus\{1\}}\{d_\gamma<\epsilon_\gamma\},~\tilde M_+:=X\setminus\tilde M_-,
$$
and set $M_+=\Gamma\backslash\tilde M_+$ and $M_-=\Gamma\backslash\tilde M_-$.
Then $M_+$ is a compact manifold with boundary and the connected components of $M_-$ are of two types (corresponding to the type of $\Gamma_{x}$):
\begin{itemize}
\item A tube: a tubular neighborhood of a short closed geodesic, homeomorphic to a ball bundle over the circle and homotopy equivalent to a circle.
\item A cusp: homeomorphic to a compact $(d-1)$-manifold times a half line.
\end{itemize}
The number of components of $M,$ is at most $C\cdot\vol(M)$ where $C=C(d)$ is a constant depending on $d$.
If $d>2$ then $M_+$ is connected.
\end{theorem}

The difference between our formulation and the standard one is that in the standard one $\epsilon=\epsilon(d)$ and hence all the $\epsilon_\gamma$ are equal to $\epsilon$. However, the proof of the statement is the same as the original one. For instance the result follows mutatis mutandis from \cite[\S 8]{BGS}.

\begin{remark}
In practice, it is useful to apply Theorem \ref{thm:think-thin} when the assignment $\gamma\mapsto\epsilon_\gamma$ is defined on a proper (conjugation invariant) subset $\Gamma'\subset\Gamma\setminus\{1\}$. This is legitimate as long as one is able to show that the resulting $M_+$ is contained in the $\epsilon$-thick part $M_{\ge \epsilon}$. Indeed, in that situation one can extend the assignment to $\Gamma\setminus\{1\}$ by declaring $\epsilon_\gamma=\epsilon$ for every $\gamma\notin\Gamma'$, without changing $M_+$.
\end{remark}

\begin{remark}
For a generic assignment $\epsilon_\gamma$, $\#\{\gamma:d_\gamma(x)=\epsilon_\gamma\}\le d$ for every $x\in X$, in which case $M_+$ is a manifold with corners.
\end{remark}

\subsection{Proof of Theorem~\ref{thm:simplicial-decomposition}}
We assume $M$ is a manifold as in Theorem~\ref{thm:simplicial-decomposition},
$X$ is its universal cover and $\Gamma$ its fundamental group, which acts on $X$ by deck transformations.

We aim to apply Theorem~\ref{thm:HV}.
To this end, we have to define constants $0<\epsilon<\epsilon_0$, a conjugation invariant subset $\Gamma'\subset \Gamma$ and a conjugation invariant assignment of constants $\epsilon_\gamma\in[\epsilon,\epsilon_0]$ for $\gamma\in\Gamma'$.
We will then set
$$
 \tilde M_-=\cup_{\gamma\in\Gamma'}\{ d_\gamma<\epsilon_\gamma\}\subset X
$$
and let $M_-$ be its image in $X$ and $M_+=M\setminus M_-$.
We will have to show that $M_+\subset M_{\geq\epsilon}$
and that for every point $x\notin \tilde M_-$ the group
$$
 \langle \gamma\in \Gamma':d_\gamma(x)\le 4\epsilon_\gamma\rangle
$$
is commutative.

We let $\epsilon(d)$ and $m(d)$ be as in Theorem~\ref{thm: margulis lemma}
and set
$$
 \epsilon_0=\epsilon(d)~\text{and}~\epsilon=\frac{\epsilon(d)}{4m(d)17^{d}}.
$$
Having fixed $\epsilon$, we use the shorthand notation $\Gamma_x=\Gamma_{x,\epsilon}$ and for a $\Gamma$-parabolic point $\zeta\in X(\infty)$ we let $\Gamma_\zeta$ and $N_\zeta\lhd\Gamma_\zeta$ be the corresponding subgroup defined in \S \ref{sec:normal-nilpotent}.
We let $\Gamma'$ be the union of all nontrivial loxodromic elements with minimal displacement $\le \epsilon(d)$ and all nontrivial parabolic elements which belong to some $N_\zeta$ with $\zeta\in X(\infty)$ $\Gamma$-parabolic.

Aiming at defining $\epsilon_\gamma$ we start by defining $i(\gamma)$ for each $\gamma\in\Gamma'$.
For a loxodromic $\gamma$ we set $i(\gamma)=0$.
For a parabolic $\gamma$ we let $i(\gamma)$ be its centrality rank at infinity.
Let us make this precise.
Assume $\gamma\in \Gamma'$ is parabolic and let $\zeta$ be its fixed point in $X(\infty)$.
Let $C^i(N_\zeta)$ be the upper central series of $N_\zeta$. That is $C^0(N_\zeta)$ is the center of $N_\zeta$ and $C^i(N_\zeta)/C^{i-1}(N_\zeta)$ is the center of $N_\zeta/C^{i-1}(N_\zeta)$.
We let $i(\gamma)$ be the minimal index such that $\gamma\in C^i(N_\zeta)$.

Finally, for every $\gamma\in\Gamma'$ we define
$$
 \epsilon_\gamma=\frac{\epsilon(d)}{4\cdot17^{i(\gamma)}}.
$$
Note that $\Gamma'$ and the assignment $\Gamma'\ni \gamma\mapsto\epsilon_\gamma$ are
invariant under conjugation in $\Gamma$.
We set
$$
 \tilde M_-=\cup_{\gamma\in\Gamma'}\{d_\gamma<\epsilon_\gamma\},~
 M_-=\Gamma\backslash \tilde M_-~\text{and}~M_+=M\setminus M_-.
$$
We now argue to show that $M_+\subset M_{\geq\epsilon}$.
Fix $x$ in the preimage of $M_+$, that is $x\notin \tilde M_-$, and $\gamma\in\Gamma\setminus \{1\}$.
We need to show that $d_\gamma(x)\geq \epsilon$.
If $\gamma$ is loxodromic this follows immediately from the fact that $x\notin \{d_\gamma<\epsilon_\gamma\}$ since $\epsilon<\epsilon_\gamma$.
Thus we may suppose by negation that $d_\gamma(x)<\epsilon$ where $\gamma$ is parabolic, fixing a point $\zeta$ at infinity.
Since $\gamma\in \Gamma_x\le\Gamma_\zeta$, for some $k\leq m(d)$, $\gamma^k\in N_\zeta$. It follows that
 $$
  d_{\gamma^{k}}(x)\le m(d)d_\gamma(x)<m(d)\epsilon=\frac{\epsilon(d)}{4\cdot 17^{d}}\le \frac{\epsilon(d)}{4\cdot 17^{i(\gamma^k)}}=\epsilon_{\gamma^{k}}.
 $$
Thus $x\in \{d_{\gamma^k}<\epsilon_{\gamma^k}\}\subset \tilde M_-$ and we get the desired contradiction.

Next we fix $x\notin\tilde M_-$ and argue to show that the group
$
 \langle \gamma\in \Gamma':d_\gamma(x)\le 4\epsilon_\gamma\rangle
$
is commutative.
To this end it is enough to show that every $\alpha,\beta$ in the generating set
$\{\gamma\in \Gamma':d_\gamma(x)\le 4\epsilon_\gamma\}$
commute.
We fix $\alpha,\beta$ in this set and assume by negation that
$[\alpha,\beta]=\alpha\beta\alpha^{-1}\beta^{-1}\ne 1$.
Since $4\epsilon_\alpha,4\epsilon_\beta\leq \epsilon(d)$ we have in particular that $\alpha,\beta\in \Gamma_{x,\epsilon(d)}$.
By assumption $\alpha$ and $\beta$ do not commute hence $\Gamma_{x,\epsilon(d)}$ is not abelian.
According to Theorem~\ref{thm: margulis lemma} we therefore have that $\Gamma_{x,\epsilon(d)}$ is parabolic.
Thus $\alpha,\beta\in \Gamma_\zeta$ for some $\zeta\in X(\infty)$. Recalling that $N_\zeta=\Gamma'\cap \Gamma_\zeta$, we deduce furthermore that $\alpha,\beta\in N_\zeta$.
Since $N_\zeta$ is nilpotent,
$$
 i([\alpha,\beta])\le\min\{ i(\alpha),i(\beta)\}-1.
$$
Therefore, assuming that $i(\alpha)\le i(\beta)$, we get
$$
 d_{[\alpha,\beta]}(x)\le 2d_\alpha(x)+2d_\beta(x) \leq
8(\epsilon_\alpha+\epsilon_\beta)\le \frac{16\epsilon(d)}{4\cdot 17^{i(\alpha)}}
< \frac{\epsilon(d)}{4\cdot 17^{i([\alpha,\beta])}}=\epsilon_{[\alpha,\beta]},
$$
contradicting the assumption that $x$ is not contained in $\tilde M_-$.

We therefore meet the conditions of Theorem~\ref{thm:HV}.
We let $D$ and $c$ be the constants given there which depend only on $d$, $\epsilon_0$ and $\epsilon$.
In fact, $D$ and $c$ depend only on $d$, as $\epsilon$ and $\epsilon_0$ were defined by means of the Margulis constants $\epsilon(d)$ and $m(d)$.
We conclude that
$(M_+,\partial M_+)$ is homotopy equivalent to a $(D,c\cdot\vol{M_+})$ simplicial pair
$(\mathcal{R},\mathcal{R}_0)$.

The statement regarding the components of the complement $M_-=M\setminus M_+$ is given in Theorem~\ref{thm:think-thin}.
Thus, we have completed the proof of Theorem~\ref{thm:simplicial-decomposition}.\qed

\section{Homology and homotopy in dimension $\ne 3$} \label{sec:hom}

\subsection{Torsion homology bounds from simplicial approximations}

A useful tool for bounding the torsion homology of simplicial complexes is the following lemma attributed to Gabber (see e.g.~\cites{soule, emery, sauer}).
A proof can be found in a paper by Soul{\'e}~\cite{soule}*{Lemma~1}.

\begin{lemma}\label{lem: soule}
	Let $A$ and $B$ be finitely generated free $\bbZ$-modules. Let
	$a_1,\ldots, a_n$ and $b_1,\ldots, b_m$ be $\bbZ$-bases of
	$A$ and $B$, respectively. We endow
	$B_\bbC=\bbC\otimes_\bbZ B$ with the Hilbert space
	structure for which $b_1,\ldots, b_m$ is a Hilbert basis.
	Let $f:A\to B$ be a homomorphism. Let $I\subset\{1,\ldots, n\}$ be a
	subset such that $\{f(a_i)\mid i\in I\}$ is a basis of $\im(f)_\bbC$.
	Then
	\[
		|\tors\coker(f)|\le \prod_{i\in I} \norm{f(a_i)}.
	\]
\end{lemma}

From the previous lemma we deduce an estimate of the torsion in relative homology groups analogous to the absolute case as in (cf.~e.g.~\cites{soule, emery, sauer}):

\begin{lemma}\label{lem: torsion bound general from simplicial pair}
For $p, D\in\bbN$ there is a constant $C(D,p)>0$ with the following property.
Let $Y$ be a space, and let $K\subset Y$ be a (possibly empty) subspace so that $(Y,K)$ is
homotopy equivalent to a $(D,V)$-simplicial pair. Then
\[ \log\bigl(|\tors H_p(Y, K;\bbZ)|\bigr)\le C(D, p)\cdot V. \]
\end{lemma}

\begin{proof}
	Note that a $(D,V)$-simplicial complex has at most $D^{p}\cdot V$ $p$-simplices.
	The relative simplicial chain complex has hence a $\bbZ$-basis of size at most $D^{p}\cdot V$ (corresponding to simplices whose boundary is not completely in the subcomplex). The $\bbZ$-basis induces a Hilbert basis on the
	complexification. The norm of the differential of a $p$-simplex is at most $(p+1)$. The statement follows now from the previous lemma.
\end{proof}

\subsection{Conclusion of proofs of Theorems~\ref{thm: main} and~\ref{thm: counting}}\label{sub: conclusion of proofs}

\begin{proof}[Proof of Theorem~\ref{thm: main}]
For $d=2$ the statement follows from the Gauss--Bonnet theorem. Fix $d\geq 4$. 
Let $M$ be a complete $d$-dimensional
manifold of finite volume with normalized bounded negative curvature. According to Theorem~\ref{thm:simplicial-decomposition} there are constants $D=D(d)>0$ and $c=c(d)>0$ and a $d$-dimensional compact submanifold $M_+\subset M$ with boundary $M_0=\partial M_+$ such that
$(M_+, M_0)$ is homotopy equivalent to a $(D, c\vol(M))$-simplicial pair.
The closure $M_-$ of the complement of $M_+$ in $M$ decomposes into compact components which are
ball bundles over circles and non-compact components. Let $M_-^c$ and $M_-^\infty$ be the union of the compact and the non-compact components, respectively. We consider the following subspaces:
	\begin{align*}
		M_-&\phantom{:}=M_c^-\cup M^\infty_-=\overline{M\backslash M_+}\\
		M_0 &\phantom{:}= \partial M_+\\
		M_0^\infty&:=M_-^\infty\cap M_0,\\
		M_0^c&:=M_-^c\cap M_0,\\
		M_c &:= M_+\cup M_-^c.
\end{align*}
\begin{figure}
\centering	
\end{figure}
%
%
Note that $(M_+, M_0^c)$ is also homotopy equivalent to
a $(D, c\vol(M))$-simplicial pair (see Remark~\ref{rem: boundary restriction}).
The space $M_c$ is also a $d$-dimensional manifold with (possibly empty) boundary.
The inclusion $M_0^\infty\subset M_-^\infty$ is a strong deformation retract which induces a strong deformation
retract $M_c\subset M$.

By Lemma~\ref{lem: torsion bound general from simplicial pair} there is a constant $C=C(d)>0$ such that
\begin{align}\label{eq: torsion bound for thick part}
	\log|\tors H_p(M_+;\bbZ)|&\le C\cdot\vol(M)\\
	\log|\tors H_p(M_+,M_0^c;\bbZ)|&\le C\cdot\vol(M)\notag
\end{align}	
in all degrees~$p$.

We claim that the inclusion $M^c_0\subset M_-^c$ is $(d-2)$-connected: $M_-^c$ is a
topological sum of ball bundles over $S^1$, thus the inclusion fits into commutative diagram with horizontal
fiber sequences:
\begin{equation*}
	\begin{tikzcd}
		\bigsqcup S^{d-2}\arrow[d, hook]\arrow[r, hook] & M_0^c\arrow[d, hook]\arrow[r] & \bigsqcup S^1\arrow[d, equal]\\
		\bigsqcup D^{d-1}\arrow[r, hook]    & M^c_-\arrow[r] & \bigsqcup S^1
	\end{tikzcd}
\end{equation*}
The long exact sequence in homotopy groups and the $5$-lemma imply the connectivity claim. Next we apply
the homotopy excision theorem~\cite{tomdieck}*{Proposition~6.10.1 on p.~152} to the pushout square
\begin{equation*}
\begin{tikzcd}
M^c_0\arrow[r, hook]\arrow[d, hook] & M_+\arrow[d, hook]\\
M_-^c\arrow[r, hook]  & M_c
\end{tikzcd}
\end{equation*}
The left vertical map is a cofibration since it is the inclusion of
a union of boundary components into a manifold with boundary.
The left vertical map is $(d-2)$-connected. By homotopy excision the right vertical map
is also $(d-2)$-connected. Since the inclusion $M_c\subset M$ is a deformation retract we obtain that the inclusion $M_+\subset M$ is $(d-2)$-connected. Hence
\begin{equation}\label{eq: from analysis of thick part}
 \log|\tors H_p(M;\bbZ)|=\log|\tors H_p(M_+;\bbZ)|\le C\cdot\vol(M)
 \end{equation}
for $p\in\{0,\ldots, d-3\}$. Next we assume that $p\ge d-2\ge 2$.

Since $M_-^c$ is homotopically a 1-dimensional complex, the long exact homology sequence of $(M_c, M_-^c)$ implies
that $H_p(M_c; \bbZ)\to H_p(M_c, M_-^c;\bbZ)$ is injective. By excision the latter module is isomorphic
to $H_p(M_+, M_0^c;\bbZ)$ for which we have the torsion bound~\eqref{eq: torsion bound for thick part}.
This concludes the proof.
\end{proof}

\begin{remark}\label{rem: main thm for closed manifolds}
	The proof above is easier if $M$ is closed since we do not have to rely on the relative statement in Theorem~\ref{thm:simplicial-decomposition}. In this case we conclude~\eqref{eq: from analysis of thick part} just from an analysis of the thick part. Then we appeal to Poincar\'{e} duality and the universal coefficient theorem to finish the proof of Theorem~\ref{thm: main}.  
\end{remark}

\begin{proof}[Proof of Theorem~\ref{thm: counting}] 
	We retain the setting of the previous proof. Let $M$ be a $d$-dimensional complete Riemannian manifold of normalized bounded negative curvature and volume $\le v$.
	
	We showed in the previous proof that $M_+\hookrightarrow M$ is a $\pi_1$-isomorphism. Since
	$M_+$ is homotopy equivalent to a $(D,CV)$-simplicial complex and the logarithm of the number of possible $2$-skeleta of such simplicial complexes is bounded by
	$\operatorname{const} V\log V$ (cf.~\cite{counting}*{Proposition on p.~1164}), the upper
	bound in the statement of Theorem~\ref{thm: counting} follows. The lower bound comes from considering hyperbolic manifolds alone and
	was already established in~\cites{counting,BGLS,commensurable}.
\end{proof}

\section{Torsion in dimension 3} \label{sec:dim3}

This section is devoted to the study of torsion in the first homology groups of 3-dimensional hyperbolic manifolds.
\S\ref{sub: bounded volume unbounded torsion}
 is devoted to the proof of Theorem~\ref{thm:3D}.
In the proceeding sections we give basic facts concerning the space of Invariant Random Subgroups
of $\PSL_2(\mathbb{C})$ and the Benjamini--Schramm of hyperbolic $3$-manifolds,
which are necessary ingredients of the proofs of Theorem~\ref{thm: Benjamini-Schramm and explosive torsion} and Theorem~\ref{thm:densetorsion}.
In  \S\ref{sec:Farber+torsion} we prove Theorem~\ref{thm: Benjamini-Schramm and explosive torsion}.
In \S\ref{sec:densetorsion} we review a construction due to Brock and Dunfield \cite{BrDu}
and prove Theorem~\ref{thm:densetorsion}.
Lastly, in \S\ref{subsection:analytic} we discuss analytic torsion.

\subsection{Bounded volume but unbounded torsion -- Proof of Theorem~\ref{thm:3D}}\label{sub: bounded volume unbounded torsion}

In this subsection we prove Theorem \ref{thm:3D}.
The manifolds $M_{p,q}$ are all obtained by different Dehn fillings of a knot complement of a fixed knot in $S^3$.

Let us first recall Dehn fillings. Let $M$ be a non-compact complete hyperbolic $3$-manifold. We assume for notational simplicity that $M$ has only one cusp. A typical example of such $M$ is a hyperbolic knot complement.
Let $M^c$ be a compact core of $M$ obtained by chopping off horoball cusp neighborhoods. The boundary $\partial M^c$ is a torus and its intrinsic metric is flat.
Let $f\colon S^1\times S^1\to \partial M^c$ be an isometry with respect to a suitably scaled flat metric on $S^1\times S^1$.
The manifold~$M$ is homeomorphic to $M^c\cup_{f} (T^2\times [0,\infty))$.
Let $\alpha$ be a closed simple geodesic in $\partial M^c$. Let $f_\alpha\colon S^1\times S^1\to \partial M^c$ be a diffeomorphism that maps $S^1\times\{*\}$ to $\alpha$. Then $M_\alpha=M_c\cup_{f_\alpha} D^2\times S^1$ is a closed manifold; we say $M_\alpha$ is \emph{obtained from $M$ by a Dehn filling along ~$\alpha$}. By the $2\pi$-theorem of Gromov--Thurston,
    $M_\alpha$ admits a complete hyperbolic structure if the length of
    $\alpha$ is greater than $2\pi$
(see~\cite{dehnfilling} for a detailed proof). A similar discussion applies if $M$ has more than one cusp.

The next lemma describes the effect of Dehn fillings on the first homology.

\begin{lemma}\label{lem: homology dehn filling}
Let $M$ be a complete hyperbolic $3$-manifold of finite volume that has exactly one cusp. Let $(\mu, \lambda)$ be a basis of $H_1(\partial M^c;\bbZ)$ such that $\lambda$ is in the kernel of $H_1(\partial M^c;\bbZ)\to H_1(M^c;\bbZ)$. Let $\alpha_{p,q}$ be a simple closed geodesic representing $p\mu+q\lambda\in H_1(\partial M^c;\bbZ)$.
    The manifold $M_{(p,q)}$ obtained from Dehn filling along $\alpha_{p,q}$ satisfies
\[ |\tors H_1(M_{(p,q)};\bbZ)|\ge p,\]
with equality in the case that $H_1(M,\mathbb{Z})\simeq \mathbb{Z}$ and it is generated by the image of  $H_1(\partial M^c;\bbZ)$.
In particular, an equality holds in case
$M$ is a knot complement.
\end{lemma}

\begin{proof} The image $\mu'\in H_1(M^c;\bbZ)$ of $\mu$ generates
an infinite cyclic subgroup since the rank of the image of $H_1(\partial M^c;\bbZ)$ in $H_1(M^c;\bbZ)$ is~$1$ by Poincare duality (``half lives half dies'', see~\cite{hatcher-notes}*{Lemma~3.5}). In the case of a knot complement $\mu'$ is the generator of $H_1(M^c;\bbZ)$ by Alexander duality.
The Mayer--Vietoris sequence implies that $H_1(M_{\alpha_{p,q}};\bbZ)$ is the cokernel of the homomorphism
    \[\operatorname{incl}_\ast\oplus (f_{\alpha_{p,q}})_\ast \colon H_1(S^1\times S^1) \to H_1(D^2 \times S^1)\oplus H_1(M^c).\]
    This cokernel maps injectively into the cokernel of the map $\operatorname{incl}_\ast\oplus (f_{\alpha_{p,q}})_\ast$ restricted to its image
    $H_1(D^2\times S^1)\oplus \langle \mu'\rangle$. The latter map is represented by a matrix $A=\begin{pmatrix}
    	0 & p\\ 1 & \ast\end{pmatrix}$. The size of the cokernel of $A$ is $|\det(A)|=p$.
\end{proof}

\begin{proof}[Proof of Theorem~\ref{thm:3D}]
If $M$ is a knot complement $S^3-K$ then
$$
  H_1(M;\bbZ)\simeq H^1(S^1)\simeq\bbZ
$$
by Alexander duality. The compact core $M^c$ is homeomorphic to $S^3$ with a tubular neighborhood of $K$ removed.
The first homology $H_1(\partial M^c;\bbC)$ has a basis $(\mu, \lambda)$ where $\mu$, the \emph{meridian}, represents the boundary
of a disk in the solid torus around~$K$, and $\mu$ is the \emph{longitude} being nullhomologous in $M$.
Let us consider the specific case of the figure eight knot $K_8$.
The knot complement $M_8=S^3-K_8$ is an arithmetic manifold, associated with an index 12 subgroup of $\PSL_2(\bbZ[\omega])$, where $\omega$ is a primitive cube root of unity, and the volume of $M_8$ is twice the volume of a regular ideal simplex which is
\[
 \vol(M_8)=6\int_0^{\pi/3}-\log(2\sin \theta)d\theta <2.03.
\]
According to Thurston~\cite{thurston}*{Theorem~4.7} every closed manifold $M_{(p,q)}$ with coprime $p,q$ obtained from Dehn filling $M_8$ along $p\mu+q\lambda$ is hyperbolic except possibly for the values \[(\pm p,\pm q)\in \{ (1,0), (0,1), (1,1), (2,1), (3,1), (4,1)\}.\]
If so, the volume of $M_{(p,q)}$ is less than the volume of $M_8$ by
Thurston's hyperbolic Dehn filling theorem (see~\cite{boileau+porti}*{Appendix~B} for a detailed proof).
By Lemma~\ref{lem: homology dehn filling} applied to $M_8$, one gets that $H_1(M_{p,q};\bbZ)=\bbZ/p$.
Taking $p\in \bbZ$ arbitrary and $q$ coprime to $p$ with the exception of the finite list above we obtain the manifolds $M_{p,q}$ in the statement of Theorem~\ref{thm:3D}.
\end{proof}


\subsection{Invariant Random Subgroups and the Benjamini--Schramm space}
Let us recall the Benjamini--Schramm space $\text{BS}(\mathbb{H}^3)$ associated to $\mathbb{H}^3$. We refer the reader to~\cites{7s,IRS} for details.
A point in BS$(\mathbb{H}^3)$ is a random complete hyperbolic $3$-manifold with a special point and a choice of frame at its tangent space. A finite volume complete hyperbolic $3$-manifold $M$ corresponds to a point in BS$(\mathbb{H}^3)$ by normalizing the Riemannian volume of $M$ and picking a point in the frame bundle over $M$ at random. One may define the topology on BS$(\mathbb{H}^3)$ directly by defining an appropriate Gromov--Hausdorff topology on the space of framed manifolds and consider the space of probability measures on that. A quicker way however is to associate it with the space of invariant random subgroups of $\PSL_2(\mathbb{C})=\text{Isom}(\mathbb{H}^3)^\circ$.

Consider the space $\text{Sub}(\PSL_2(\mathbb{C}))$ of closed subgroups of $\PSL_2(\mathbb{C})$ equipped with the Chabauty topology. An \emph{IRS} on $\PSL_2(\mathbb{C})$ is a conjugation invariant Borel regular probability measure on $\text{Sub}(\PSL_2(\mathbb{C}))$. An IRS is said to be \emph{discrete} if it is supported almost surely on the set of discrete subgroups. We let $\text{IRS}_d(\PSL_2(\mathbb{C}))$ denote the space of discrete IRS on $G$ equipped with the weak-$*$ topology. Then $\text{IRS}_d(G)$ is compact (see \cite{WUD} or \cite[Sec. 3.2]{IRS}). By fixing an origin and a tangent frame in $\mathbb{H}^3$, we obtain a map from the set of discrete subgroups $\text{Sub}_d(\PSL_2(\mathbb{C}))$ to the space of framed hyperbolic $3$-manifolds, $\Gamma\mapsto \Gamma\backslash\mathbb{H}^3$. Thus every IRS on $\PSL_2(\mathbb{C})$ defines a point in BS$(\mathbb{H}^3)$ via pushing forward the measure. This map is one to one and its image can be characterized as the set of points in BS$(\mathbb{H}^3)$ which are invariant under the geodesic flow, denoted $\text{BS}(\mathbb{H}^3)^\text{inv}$.
The inverse of this map is defined via taking the deck-transformations associated to the fundamental group of a random manifold. These maps identify the compact topological space $\text{IRS}_d(\PSL_2(\mathbb{C}))$ with $\text{BS}(\mathbb{H}^3)^\text{inv}$, and the weak-$*$ topology on the first is the \emph{Benjamini--Schramm topology} (or short: \emph{BS-topology}) on the latter.

\subsection{Benjamini-Schramm convergent sequences of manifolds}
The trivial IRS --- the Dirac mass $\delta_{\langle 1\rangle}$ at the identity of $\PSL_2(\mathbb{C})$, corresponds to the random manifold which is almost surely $\mathbb{H}^3$.\footnote{Since the frame bundle over $\mathbb{H}^3$ is homogeneous, we may allow ourselves to omit the special point and the frame from the random manifold associated to $\delta_{\langle 1\rangle}$ and simply denote it by $\mathbb{H}^3$.} A sequence of complete finite volume hyperbolic manifolds $M_n$ BS-converges to $\bbH^3$ if and only if
   the corresponding sequence of invariant random subgroups $\mu_n$ converges to $\delta_{\langle 1\rangle}$.
    Recall the following characterization from~\cite{7s}:

\begin{lemma}
A sequence $M_n$ of complete hyperbolic $3$-manifolds of finite volume
BS-converges to $\bbH^3$ if and only if for every $r>0$
\[
 \frac{\vol\{x\in M_n\mid \operatorname{InjRad}_{M_n}(x)>r\}}{\vol{(M_n)}}\longrightarrow 1.
\]
\end{lemma}

\subsection{Metric on BS$(\mathbb{H}^3)$}

It is well known that $\text{Sub}(\PSL_2(\mathbb{C}))$ is metrizable. The following elegant way to define a metric was suggested by Ian Biringer~\cite{biringer}. Let $B_r$ denote the $r$ ball around the identity in $\PSL_2(\mathbb{C})$ with respect to the matrix norm. We let Hd denote the Hausdorff distance between bounded sets in $\PSL_2(\mathbb{C})$. The distance $\rho$ on $\text{Sub}(\PSL_2(\mathbb{C}))$ is defined as follow:
$$
 \rho(H_1,H_2):=\int_{r>0}\text{Hd}(H_1\cap B_r,H_2\cap B_r)e^{-r}dr.
$$
As a consequence, also the space $\text{BS}(\mathbb{H}^3)^\text{inv}\simeq\text{IRS}_d(\PSL_2(\mathbb{C}))$ is metrizable.
A concrete metric on it is given by the Kantorovich--Wasserstein metric (associated with the bounded metric $\min\{\rho,1\}$),
given by the formula
$$\text{IRS}_d(\PSL_2(\mathbb{C})) \ni \mu,\nu \mapsto \inf_{\eta\in J(\mu,\nu)} \int \min\{\rho(H_1,H_2),1\} d\eta(H_1,H_2), $$
where $J(\mu,\nu)$ is the space of joinings of $\mu$ and $\nu$, i.e. the space of probability measures on $\text{Sub}(\PSL_2(\mathbb{C}))^2$ with marginal measures $\mu$ and $\nu$.

\subsection{A simple construction of a sequence with explosive torsion}\label{sec:Farber+torsion}

\begin{theorem}\label{thm: Benjamini-Schramm and explosive torsion}
There exists a sequence of closed hyperbolic 3-manifolds $M_n$
that converges in the Benjamini--Schramm topology to $\bbH^3$ such
that
\[ \lim_{n\to\infty} \frac{\log|\tors~ H_1(M_n,\bbZ)|}{\vol(M_n)}=\infty. \]
Furthermore, for any function $f\colon (0,\infty)\to (0,\infty)$
there is such a 	sequence $M_n$ with
$$
 \log|\tors~ H_1(M_n,\bbZ)|>f({\vol(M_n)}).
$$
\end{theorem}

Denote by $\mathcal{F}$ the set consisting of complete finite volume hyperbolic $3$-manifolds viewed as a subset of $\text{BS}(\mathbb{H}^3)^\text{inv}$.
Denote by $\mathcal{K}$ the subset of $\mathcal{F}$ consisting of compact manifolds
and set $\mathcal{U}=\mathcal{F}-\mathcal{K}$.
Denote by $\mathcal{U}_1$ the subset of $\mathcal{U}$ consisting of manifolds with exactly one cusp
and for $C>0$ denote
\[ \mathcal{K}(C)= \{M\in \mathcal{K}: \log|\tors~H_1(M;\bbZ)| \geq C\vol(M)\}. \]

The proof of Theorem~\ref{thm: Benjamini-Schramm and explosive torsion} will follow from the following proposition.

\begin{prop} \label{prop:closure properties}
Let $\clos(\_)$ denote the closure of a set in the BS-topology. The BS-topology satisfies the following properties:
\begin{enumerate}
\item $\mathbb{H}^3\in \clos(\mathcal{U})$.
\item $\clos(\calU_1)=\clos(\calU)$.
\item For every $C>0$, $\mathcal{U}\subset \clos(\mathcal{K}(C))$.
\end{enumerate}
\end{prop}

\begin{proof}[Proof of Theorem~\ref{thm: Benjamini-Schramm and explosive torsion} from Proposition \ref{prop:closure properties}]
Let $d$ a metric that induces the BS-topology on $\text{BS}(\mathbb{H}^3)^\text{inv}$.
For every $n\in\bbN$,
pick $M'_n\in\mathcal{U}$ with $d(M'_n,\mathbb{H}^3)<1/n$ according to~(1), and
pick $M_n\in \mathcal{K}(f(n))$ with $d(M_n,M'_n)<1/n$ according to~(3).
The sequence $(M_n)$ will have the desired properties.
\end{proof}

\begin{lemma}\label{lem:deff->BC}
Let $\Gamma_0$ be a lattice in $\PSL_2(\mathbb{C})$ and let $f_n\colon\Gamma_0\to\PSL_2(\mathbb{C}),~n\in\mathbb{N}$ be homomorphisms such that $\Gamma_n:=f_n(\Gamma_0)$ are lattices, and $f_n$ converges to the inclusion in the topology of $\text{Hom}(\Gamma_0,\PSL_2(\mathbb{C}))$. Let $M_n=\Gamma_n\backslash\mathbb{H}^3$ for $n\in\bbN$. Then $M_n$ converges to $M_0$ in the BS-topology.
\end{lemma}

\begin{proof}
This is a consequence of \cite[Proposition 11.2]{local-rigidity}. Note that \cite[Proposition 11.2]{local-rigidity} is formulated for uniform lattices, but the proof given in \cite{local-rigidity} applies with no changes to the non-uniform case. Indeed, the proof relies on the fact that $\vol(M_n)\to\vol(M_0)$ which is proved in \cite{local-rigidity} for uniform lattices obtained by arbitrary deformation, and is valid in our situation due to Thurston's hyperbolic Dehn filling theorem.
\end{proof}

\begin{proof}[Proof of Proposition~\ref{prop:closure properties}]
(1) Any residual tower of non-compact hyperbolic $3$-manifolds of finite volume BS-converges to $\mathbb{H}^3$.
For concreteness, we could take $M_n=\mathbb{H}^3/\Gamma_n$ where $\Gamma_n$
is the kernel of $\SL_2(\mathbb{Z}[i])\to \SL_2(\mathbb{Z}/n[i])$. Obviously $M_n\to \mathbb{H}^3$ (see \cite[Section 5]{7s} for concrete estimates on the rate of convergence of congruence towers).

(2) Suppose that $M=\Gamma\backslash \mathbb{H}^3$ is of finite volume and non-compact. The boundary of the compact core $M^c$ of $M$
consists of $n$ tori $T_1, \ldots, T_n$. Fix bases $(\lambda_i, \mu_i)$ of $H_1(T_i;\bbZ)$. For $k>1$ we perform
Dehn fillings in the first $(n-1)$ cusps along curves representing
$\lambda_i+k\mu_i$ to obtain a manifold $M_k$. For $k$ large enough,
$M_k$ admits a complete hyperbolic structure of finite volume
by Thurston's hyperbolic Dehn filling theorem. Hence $M_k\in\calU_1$.
Further, $M_k$ tends to $M$ as $k\to \infty$ in the topology of the representation variety by the same theorem. By Lemma~\ref{lem:deff->BC} $M_k$ BS-converges to $M$.

(3) By part (2) it is enough to show that a complete hyperbolic
$3$-manifold $M$ with one cusp
is in the closure of $\mathcal{K}(C)$. By~\cite{hatcher-notes}*{Lemma~3.5} there is a basis $(\lambda,\mu)$ of $H_1(\partial M^c;\bbZ)$ where $M^c$ is
    	the compact core of $M$ such that the hypothesis of Lemma~\ref{lem: homology dehn filling} is satisfied. Performing
    	Dehn fillings along $\lambda+k\mu$ we obtain closed
    	manifolds $M_k$.
    	By Thurston's hyperbolic Dehn filling theorem $M_k$ is
    	hyperbolic for $k$ large enough and $\vol(M_k)\le \vol(M)$.
    	As in (2) we also conclude that $M_k\to M$ in the BS-topology.
    	By Lemma~\ref{lem: homology dehn filling},
    	\[ |\tors H_1(M_k;\bbZ)|/\vol(M_k)\ge k/\vol(M)\to \infty \text{ as $k\to \infty$.}\qedhere\]
 \end{proof}

\subsection{Asymptotic density of the normalized torsion} \label{sec:densetorsion}

In this subsection we use a result of Brock and Dunfield \cite{BrDu} in order to modify the construction given in Subsection~\ref{sec:Farber+torsion}.
The following theorem is extracted from \cite{BrDu}.

\begin{theorem} [{\cite[\S2]{BrDu}}] \label{thm:BD}
There exists a sequence $M_n$ consisting of finite volume hyperbolic
3-manifolds which BS-converges to $\mathbb{H}^3$ such that for each $n$, $M_n$ has one cusp and $H_1(M_n,\mathbb{Z})\cong \mathbb{Z}$.
Furthermore the first homology of the cusp surjects on  $H_1(M_n,\mathbb{Z})$.
\end{theorem}

Unfortunately, neither the theorem, nor the sequence $M_n$ appear explicitly in \cite{BrDu}.
We explain below how to modify the construction that does appear there, in order to prove Theorem~\ref{thm:BD}.
We follow closely the construction given in {\cite[\S2.4]{BrDu}}.
We advise the reader to keep this paper close. Further details and justifications could be found there.


\begin{proof}[Proof of Theorem~\ref{thm:BD}]
We start by fixing once and for all a Heegaard splitting of $S^3$, $S^3=H^+\cup S\cup H^-$, where $H^\pm$ are open and $S=\partial H^+=\partial H^-$ is of genus 2
\footnote{For comparison sake: in the notation of \cite[\S2.4]{BrDu}, we specialize here to the case $N_0=S^3$, $g=2$ and $A=\{0\}$.}.
We will identify below the inclusion of $S$ in a neighborhood of it in $S^3$ with the inclusion
of $S\times \{0\}$ inside $S\times [0,6]$.
We also fix a pants decomposition $P$ of $S$ so that the pared manifolds $(H^\pm,P)$ are acylindrical
and a
separating essential simple closed curve $\gamma$ on $S$ so that the pared manifold
\[ U=\left( (S\times [0,2]-(\gamma\times \{1\}),P\times \{0,2\}\right) \]
is acylindrical.

The reminder of the construction will be dependent on a parameter $R>0$ which we now fix.
In our notation below we will stress the dependence on $R$, which is implicit in \cite[\S2.4]{BrDu}.
We pick a pseudo-Anosov diffeomorphism $f(R):S\to S$ such that the corresponding mapping torus $M_{f(R)}$ has injectivity radius larger than $R+1$
and define a family of links $L_n(R)$ which lie in $S\times [0,6]$ by
\[ L_n(R)= P\times \{1\} \cup f(R)^n(P)\times \{2\} \cup f(R)^n(\gamma)\times \{3\} \cup f(R)^n(P)\times \{4\} \cup P\times \{5\}.\]
By \cite[Lemma~2.6]{BrDu}, for any given $R$, for any large enough $n$ (depending on $R$), the manifold $S^3-L_n(R)$ has a finite volume hyperbolic manifold structure.
Assuming $n$ is indeed large enough, we denote this hyperbolic manifold by $N_n(R)$.
By \cite[Lemma~2.6]{BrDu} we also get that
\[ \lim_{n\to \infty} \frac{\vol(N_n(R)_{<R})}{\vol(N_n(R))}=0. \]
We now perform a diagonalizing argument: for every $m\in \mathbb{N}$ we fix $n_m\in\mathbb{N}$ so that
\[ \frac{\vol(N_{n_m}(m)_{<m})}{\vol(N_{n_m}(m))}<\frac{1}{m} \]
and set $N'_m=N_{n_m}(m)$.
We conclude that the sequence $N'_m$
BS-converges to $\mathbb{H}^3$ as $m$ tends to $\infty$.

For fixed $m$ and $k$ define the manifold $M_{m,k}$ obtained from $S^3$ by performing $1/k$ Dehn filling along the links in $L_{n_m}(m)$ of height 1 and 2 and $-1/k$ Dehn filling along the links in $L_{n_m}(m)$ of height 4 and 5.
If we further make a $1/k$ Dehn filling along the link at height 3 we would get a  manifold (denoted $N_{n_m,k}$
in \cite[\S2.4]{BrDu} for the implicit fixed parameter $R=m$), which is an integral homology sphere by
\cite[Lemma~2.5]{BrDu}.
We do not perform this last Dehn filling!
Thus $H_1(M_{m,k},\mathbb{Z})\simeq \mathbb{Z}$ by Alexander duality,
and the homology of the cusp surjects on it.

Fixing $m$, by Thurston's Theorem we get that for $k$ large enough $M_{m,k}$ has the structure of a finite volume hyperbolic manifold.
Thus, for every $m$ and for a large enough $k$ (depending on $m$), the manifold $M_{m,k}$
is a finite volume hyperbolic
3-manifolds which
has one cusp and $H_1(M_{m,k},\mathbb{Z})\simeq \mathbb{Z}$ is generated by this cusp.
Thurston's Theorem also tells us that for a fixed $m$, when $k$ tends to $\infty$, $M_{m,k}$ tends to $N'_m$ in the representation variety topology.
As explained in Lemma~\ref{lem:deff->BC}, by \cite[Proposition 11.2]{local-rigidity} we get that $M_{m,k}$ also BS-converges to $N'_m$ as $k$ tends to $\infty$.
By the fact that $N'_m$ itself BS-converges to $\mathbb{H}^3$ as $m$ tends to $\infty$,
using once more a diagonal argument, this time on $m$ and $k$, we obtain the required sequence of
finite volume hyperbolic
3-manifolds which BS-converges to $\mathbb{H}^3$, each having one cusp which generates its first homology group.
\end{proof}

\begin{proof}[Proof of Theorem~\ref{thm:densetorsion}]
We fix a sequence $M_n$ as given in Theorem~\ref{thm:BD}.
In particular $v_n=\vol(M_n)$ tends to infinity.
We let $p_n=[\alpha v_n]$ and for every $q\in\mathbb{Z}$ we consider the manifold $N^q_n=(M_n)_{(p_n,q)}$ constructed in
Lemma~\ref{lem: homology dehn filling}
by performing a
 $(p_n,q)$ Dehn filling along the cusp.
By Lemma~\ref{lem: homology dehn filling}, for every $q$,
$|\tors H_1(N^q_n;\bbZ)|= p_n$ and by Lemma~\ref{lem:deff->BC} and Thurston's theorem, $N^q_n$ BS-converges to $M_n$ as $q\to \infty$.
In particular,
\[ \lim_{q\to \infty} \frac{\log|\tors~ H_1(N^q_n,\bbZ)|}{\vol(N^q_n)}=\frac{p_n}{v_n}=\frac{[\alpha v_n]}{v_n}.\]
Since $\lim_n \frac{[\alpha v_n]}{v_n}=\alpha$,
using a diagonalizing argument we may pick a sequence $M^\alpha_n=N^{q_n}_n$ with the required properties.
\end{proof}

\subsection{On analytic torsion} \label{subsection:analytic}

The \emph{Ray-Singer torsion} $\tau(M)$ of a Riemannian manifold
is defined as
\[\tau(M)=\frac{1}{2}\sum_{k=0}^{\dim M}(-1)^kk\log(\det{'}(\Delta_k))\]
where $\det{'}$ is the zeta-regularized product of eigenvalues of the Laplacian $\Delta_k$ on smooth $k$-forms. The $p$-th \emph{regulator} $R_p(M)$ is the covolume of the free part of $H^p(M;\bbZ)$ as a lattice in $H^p(M;\bbR)$ with respect to the
harmonic metric, i.e.~the metric induced from the usual scalar product of harmonic $p$-forms. By the Ray-Singer conjecture, proved by Cheeger and M\"uller, the analytic torsion coincides with the Reidemeister torsion~\cites{cheeger,mueller}. For a $3$-dimensional manifold $M$ this implies the relation
\[\tau(M) = -\log |\tors H_1 (M, \bbZ) |+\log(\vol(M))+2\log(R_1 (M)).\]
In particular, if $M$ is a rational homology sphere, $R_1(M)=0$ and we get
\[\tau(M) = -\log |\tors H_1 (M, \bbZ) |+\log(\vol(M)).\]
The following corollary is thus an immediate application of Theorem~\ref{thm:densetorsion}.

\begin{corollary}
For every $\alpha\in [0,\infty]$ there exists a sequence of closed hyperbolic 3-manifolds $M^\alpha_n$
which are all rational homology spheres, such that the sequence $M^\alpha_n$
converges in the Benjamini--Schramm topology to $\bbH^3$ and
for the Ray-Singer torsion we have
\[\frac{-\tau (M^\alpha_n)}{\vol(M^\alpha_n)}\to \alpha.\]
\end{corollary}

Note that in contrast to the above corollary, if $(M_n)$ is a residual tower of coverings of a fixed hyperbolic $3$-manifold $M$ then the proof of~\cite{bergeron+venkatesh}*{Theorem~4.5} implies that
\[\liminf_{i\to\infty} \frac{-\tau(M_n)}{\vol(M_n)}\le \frac{1}{6\pi}. \]

\end{document}